\documentclass[11pt]{amsart}

\usepackage{amsmath}
\usepackage{amsthm}
\usepackage{hyperref}
\usepackage[margin=1.1in]{geometry}

\numberwithin{equation}{section}

\newtheorem{prop}{Proposition}
\newtheorem{lemma}[prop]{Lemma}

\newtheorem{thm}[prop]{Theorem}
\newtheorem{cor}[prop]{Corollary}
\newtheorem{conj}[prop]{Conjecture}
\numberwithin{prop}{section}

\theoremstyle{definition}
\newtheorem{defn}[prop]{Definition}

\newtheorem{rmk}[prop]{Remark}

\newcommand{\del}{\partial}
\newcommand{\delb}{\bar{\partial}}\newcommand{\dt}{\frac{\partial}{\partial t}}
\newcommand{\brs}[1]{\left| #1 \right|}

\newcommand{\gG}{\Gamma}
\renewcommand{\gg}{\gamma}
\newcommand{\gD}{\Delta}
\newcommand{\gd}{\delta}
\newcommand{\gs}{\sigma}

\newcommand{\gU}{\Upsilon}

\newcommand{\gl}{\lambda}

\newcommand{\gw}{\omega}
\newcommand{\ga}{\alpha}
\newcommand{\gb}{\beta}
\renewcommand{\ge}{\epsilon}
\newcommand{\N}{\nabla}

\newcommand{\til}[1]{\widetilde{#1}}

\renewcommand{\bar}[1]{\overline{#1}}

\renewcommand{\i}{\sqrt{-1}}

\newcommand{\bg}{\bar{g}}

\newcommand{\bi}{\bar{i}}
\newcommand{\bj}{\bar{j}}
\newcommand{\bk}{\bar{k}}
\newcommand{\bl}{\bar{l}}
\newcommand{\bm}{\bar{m}}
\newcommand{\bn}{\bar{n}}
\newcommand{\bp}{\bar{p}}
\newcommand{\bq}{\bar{q}}

\newcommand{\bs}{\bar{s}}

\newcommand{\IP}[1]{\left<#1\right>}

\newcommand{\bga}{\bar{\alpha}}
\newcommand{\bgb}{\bar{\beta}}
\newcommand{\bgg}{\bar{\gamma}}
\newcommand{\bmu}{\bar{\mu}}
\newcommand{\bnu}{\bar{\nu}}
\newcommand{\bgs}{\bar{\sigma}}

\newcommand{\HH}{\mathcal{H}}

\DeclareMathOperator{\Rc}{Rc}
\DeclareMathOperator{\Rm}{Rm}

\DeclareMathOperator{\tr}{tr}

\DeclareMathOperator{\Id}{Id}

\DeclareMathOperator{\End}{End}
\DeclareMathOperator{\rank}{rank}

\begin{document}

\title[Pluriclosed flow on generalized K\"ahler manifolds with split tangent
bundle]{Pluriclosed flow on generalized K\"ahler manifolds with split tangent
bundle}

\begin{abstract} We show that the pluriclosed flow preserves generalized
K\"ahler structures with the extra condition $[J_+,J_-] = 0$, a condition
referred to as ``split tangent bundle.''  Moreover, we show
that in this case the flow reduces to a nonconvex fully nonlinear
parabolic flow of a scalar potential function.  We prove a number of a priori
estimates for this equation, including a general estimate in dimension $n=2$ of
Evans-Krylov type requiring a new argument due to the nonconvexity of the
equation.  The main result is a long time existence theorem for the flow
in dimension $n=2$, covering most cases.  We also show that the pluriclosed flow
represents the parabolic analogue to an elliptic problem which is a very natural
generalization of the Calabi conjecture to the setting of generalized K\"ahler
geometry with split tangent bundle.
\end{abstract}

\date{\today}

\author{Jeffrey Streets}

\address{Rowland Hall\\
         University of California\\
         Irvine, CA 92617}
\email{\href{mailto:jstreets@uci.edu}{jstreets@uci.edu}}

\thanks{The author was partly supported by the National Science
Foundation DMS-1301864 and an Alfred P. Sloan Fellowship.}

\maketitle

\section{Introduction}

A generalized K\"ahler structure on a compact manifold $M$ is a triple
$(g,J_+,J_-)$ of a Riemannian metric and two integrable complex structures
$J_{\pm}$ so that
\begin{align*}
d^c_+ \gw_+ =&\ - d^c_- \gw_-\\
d d^c_+ \gw_+ =&\ - d d^c_- \gw_- = 0.
\end{align*}
These equations first arose in \cite{GHR} through investigations of
supersymmetric
sigma models.  Later these equations were given a purely geometric context
through Hitchin's generalized geometric structures \cite{Gualtthesis},
\cite{HitchinGCY}
.  One interpretation of the
generalized K\"ahler condition is as a pair of ``pluriclosed" or ``strong
K\"ahler with torsion (SKT)" structures, satisfying further compatibility
conditions.  In \cite{ST1,ST2} the
author
and Tian introduced a parabolic flow of pluriclosed structures.  Later we
discovered that this flow preserves generalized K\"ahler structure, suitably
interpreted \cite{STGK}.  A crucial observation of \cite{STGK} is that in order
to
preserve generalized K\"ahler structure the complex structures $J_{\pm}$ must
themselves flow as well.  The construction of this flow will be reviewed in \S
\ref{gkflow}.

In this paper we analyze this flow in the special case that the generalized
K\"ahler structure satisfies the further condition $[J_+,J_-] = 0$.  The first
main result is that with this extra initial condition, the complex structures
$J_{\pm}$ remain fixed, and the flow moreover reduces to a flow of a scalar
potential satisfying a nonconvex fully nonlinear parabolic equation.

\begin{thm} \label{mainthm1} Let $(M^{2n}, g, J_{\pm})$ be a generalized
K\"ahler manifold satisfying $[J_+,J_-] = 0$.  Let $\gw_t$ denote the solution
to pluriclosed flow on $(M, J_+)$ with initial condition $\gw_0 = g(J_+
\cdot, \cdot)$.  Then $\gw_t$ is a generalized K\"ahler metric on $(M,
J_{\pm})$.  Moreover, in this setting the pluriclosed flow reduces to a fully
nonlinear parabolic flow of a scalar potential function $f$ (see \S
\ref{reduction}).
\end{thm}

\begin{rmk}  The evolution equation for $f$ is closely related to generalized
Calabi-Yau equations which have appeared in physics literature
(\cite{Buscher,HullSSC,HullGCY,HitchinGCY,RocekMCY}).  In a way, Theorem
\ref{mainthm1} treats the generalized K\"ahler-Ricci flow of \cite{STGK} in the
``most classical case" of generalized K\"ahler manifolds described in
\cite{GHR}.  Already in \cite{GHR} a local potential in the case $[J_+,J_-] = 0$
is described, which is essentially the one we use here.  In later works,
culminating in \cite{HullGCY}, a local potential for generalized K\"ahler
structures is shown to exist in full generality.  A fascinating direction for
future work is to reduce the generalized K\"ahler-Ricci flow of \cite{STGK} to a
flow of a potential function.
\end{rmk}

Theorem \ref{mainthm1} allows us to make a more tractable refinement of the
existence
conjecture for pluriclosed flow (\cite{ST2} Conjecture 5.2), where it is
suggested
that
the natural cohomological obstructions to the long time existence of the flow
are the
only obstructions (see Conjecture \ref{coneconj}).  Moreover, it also suggests a
very
natural analogue of Calabi's conjecture \cite{Calabi} in this setting, which
directly generalizes the classical conjecture in the K\"ahler setting, resolved
by Yau \cite{Yau}.  While we do not address this Calabi-type conjecture directly
here, it is clear that the a priori estimates we establish can be used in this
setting as well.  The second main theorem of this paper is to obtain the smooth
long
time existence of the pluriclosed flow on many generalized K\"ahler complex
surfaces, mostly confirming Conjecture \ref{coneconj}.

As it turns out, the evolution equation satisfied by the scalar potential
function of Theorem \ref{mainthm1} is a \emph{nonconvex} fully nonlinear
parabolic equation.  This means
that many of the well-developed techniques for understanding for instance the
parabolic Monge Ampere equation (arising from K\"ahler-Ricci flow) do not apply
directly to this equation.  One notable difficulty is that the delicate
techniques used in establishing the parabolic Evans-Krylov estimates
\cite{Evans,Krylov} no
longer apply, as they rely heavily on the convexity hypothesis.  Nonetheless we
are able to overcome these difficulties to provide a fairly complete picture in
the case $n=2$:

\begin{thm} \label{mainthm2} Let $(M^4, g_0, J_{\pm})$ be a generalized K\"ahler
surface satisfying $[J_+,J_-] = 0$ and $\rank_{J_{\pm}} = 1$.  Suppose $(M^4,
J_+)$ is biholomorphic to
one of:
\begin{enumerate}
\item A ruled surface over a curve of genus $g \geq 1$.
\item A bi-elliptic surface,
\item An elliptic fibration of Kodaira dimension $1$,
\item A compact complex surface of general type, whose universal cover is
biholomorphic to $\mathbb H \times \mathbb H$,
\item An Inoue surface of type $S_M$.
\end{enumerate}
In case (1) the maximal existence time of the flow is determined by
cohomological data associated to $g_0$ (see \S \ref{conesec}).  In all other
cases the solution to pluriclosed flow with initial condition $g_0$ exists on
$[0,\infty)$.
\end{thm}

\begin{rmk} Generalized K\"ahler surfaces satisfying $[J_+,J_-] = 0$ with
$\rank_{J_{\pm}} = 1$ (see Definition \ref{rankdef}) were classified by
Apostolov-Gualtieri \cite{GaultApost}, building on work of Beauville
\cite{Bville}.  The only cases we are not able to treat are ruled surfaces over
$\mathbb {CP}^1$ and the Hopf surfaces.  The reason for this is the lack of a
certain a priori estimate which requires certain conditions on a background
metric to be satisfied (see the proof in \S \ref{mainsec}).  It seems likely
that the flow always exists smoothly for the natural conjectural existence time
(see Conjecture \ref{coneconj}), and indeed many of our estimates, including the
crucial Harnack estimate, apply in full generality.
\end{rmk}

\begin{rmk} In the cases $\rank_{J_{\pm}} \in \{0,2\}$, it follows that in fact
$J_{\pm} = \pm J_-$, and generalized K\"ahler metrics are automatically
K\"ahler.  Thus solutions to pluriclosed flow in this setting are solutions
to K\"ahler-Ricci flow, where for instance the sharp long time existence result
is already known in \cite{TianZhang}.
\end{rmk}

In the course of the proof of Theorem \ref{mainthm2} we establish a number of a
priori estimates.  First we establish a general $L^{\infty}$ estimate for the
potential
function along the flow (Lemma \ref{fbound}) in any dimension.  By exploiting
special structure in the case $n=2$ we establish a uniform bound on the time
derivative of the potential function as well (Proposition \ref{fdotest}).  One
crucial step is to establish an upper bound for the metric in the presence of a
lower bound.  This requires the introduction of a kind of ``torsion potential"
which has a miraculously simple evolution equation (Lemma
\ref{torsionpotentialev}), and can be used to control the difficult torsion
terms appearing in various evolution equations.  This torsion potential has
applications to understanding pluriclosed flow in full generality, and will be
expounded upon in future work.  The final key estimate is an
Evans-Krylov type estimate for the pluriclosed flow in this setting for
dimension $n=2$ (cf. Theorem \ref{EKregularity}).

Here is an outline of the rest of this paper: in \S \ref{bckgrnd} we recall the
relationship of pluriclosed flow and generalized K\"ahler geometry, and also
give some background on generalized K\"ahler structures with $[J_+,J_-] = 0$. 
In \S \ref{pcfgkc} we establish that the pluriclosed flow preserves the
generalized
K\"ahler condition in the ``natural gauge,"  and give the proof of Theorem
\ref{mainthm1}.  We also motivate a sharp local existence conjecture
specializing (\cite{ST2} Conjecture 5.2) to this special case.  In \S \ref{eveq}
we record a number of evolution equations along the pluriclosed flow special to
this setting.  Next in \S \ref{estimates} we obtain a number of a priori
estimates on the potential function, its derivatives, and the metric tensor. 
Section \ref{harnack} has the proof of Evans-Krylov type regularity for our
equation in dimension $n=2$.  We bring these estimates together to prove Theorem
\ref{mainthm2} in \S \ref{mainsec}.
\vskip 0.1in

\noindent \textbf{Acknowledgements}: The author would like to thank Vestislav
Apostolov, Marco
Gualtieri and Jess Boling for helpful conversations on the results of this
paper.  Also, in 2012 after \cite{STGK} was written, the author had a discussion
with G. Tian, M. Rocek and C. Hull during which it was noted that the results of
\cite{GHR, HullGCY,Lind} suggested that the generalized K\"ahler-Ricci flow of
\cite{STGK} should reduce to a potential function.  Based on this the author and
Tian made preliminary investigations in this direction.  The author would like
to thank all three for these inspirational discussions.  The author also thanks
an anonymous referee for helpful comments.

\section{Background} \label{bckgrnd}
\subsection{Pluriclosed flow} \label{PCFbck}

Let $(M^{2n}, J)$ be a complex manifold.  Suppose $g$ is a Riemannian metric
compatible with $J$, with K\"ahler form $\gw = g(J \cdot, \cdot)$.  We say that
the metric $g$ is \emph{pluriclosed} if
\begin{align*}
 \del \delb \gw = 0.
\end{align*}
In \cite{ST1} the author and Tian defined a parabolic flow of Hermitian metrics
which preserves the pluriclosed condition.  To define it, let
\begin{align} \label{Pdef}
 P(\gw) := - \del \del^*_{\gw} \gw - \delb \delb^*_{\gw} \gw - \frac{\i}{2} \del
\delb \log \det g,
\end{align}
where the last term stands in for the representative of the first Chern
class arising from the Chern connection.  The pluriclosed flow is then the
evolution equation
\begin{align} \label{PCF}
 \dt \gw = - P(\gw).
\end{align}
One can phrase this flow equation using the curvature of the Chern connection. 
In particular, let $T$ denote the torsion of the Chern connection, and let $S =
\tr_{\gw} \Omega$ denote the curvature endomorphism of the Chern connection. 
Lastly, set
\begin{align} \label{Qdef}
 Q_{i \bj} =&\ g^{\bl k} g^{\bn m} T_{i k \bn} T_{\bj \bl m}.
\end{align}
It follows (\cite{ST1} Proposition 3.3) that the associated Hermitian metrics
$g_t$ satisfy
\begin{align*}
 \dt g =&\ - S + Q.
\end{align*}
We also here define the normalized flow.  We say that a one-parameter family of
pluriclosed metrics $g_t$ satisfies the \emph{normalized pluriclosed flow} if
\begin{align} \label{normflow}
 \dt \gw =&\ - P(\gw) - \gw.
\end{align}
In some applications of Ricci flow the ``normalized flow'' means to fix the
volume to be constant along the flow.  The normalization of (\ref{normflow})
generally does not preserve the volume.  However, for various applications in
K\"ahler Ricci flow the scaling choice as in (\ref{normflow}) is useful, and is
the one most useful to us here.

\subsection{Pluriclosed flow and generalized K\"ahler geometry} \label{gkflow}

We briefly recall the construction of solutions to the generalized Kahler Ricci
flow.  Let $(M^{2n}, g, J_{\pm})$ be a generalized Kahler structure.  Let
$\gw_{\pm}$ denote the associated K\"ahler forms, both of which are pluriclosed.
 Let $\gw_{\pm}(t)$ denote the solution to pluriclosed flow with these initial
conditions.  By \cite{ST2} we know that the associated metrics $g_{\pm}(t)$ and
three-forms
$H_{\pm}(t) = d^c_{\pm} \gw_{\pm}$ are solutions to the (gauge-modified)
\emph{$B$-field flow:}
\begin{gather} \label{Bflow}
 \begin{split}
 \dt g_{\pm} =&\ - 2 \Rc + \frac{1}{2} \HH + L_{\theta_{\pm}^{\sharp}} g,\\
 \dt H_{\pm} =&\ \gD_d H + L_{\theta_{\pm}^{\sharp}} H,
 \end{split}
\end{gather}
where $\HH_{ij} = H_{ipq} H_j^{pq}$.  Now let $(\phi_{\pm})_t$
denote the one-parameter families of diffeomorphisms
generated by $\theta_{\pm}^{\sharp}$, respectively.  It follows that
$\left(\phi_{\pm}\right)_t^* (g_{\pm})_t,
\left(\phi_{\pm}\right)_t^* H_{\pm}(t)$ are both solutions to the $B$-field
flow.  Moreover, one observes that $(\left(\phi_{-}\right)_t^* (g_{-})_t,
- \left(\phi_{-}\right)_t^* H_{-}(t))$ is ALSO a solution to the $B$-field flow.
 Since $g_+(0) = g_-(0)$ and $H_+(0) = - H_-(0)$, by
uniqueness of solutions to (\ref{Bflow}) one has $(g_+(t), H_+(t)) = (g_-(t), -
H_-(t))$ for all
$t$ such that the flows exist.  Indeed, it furthermore follows that $g$ is
compatible with $(J_{\pm})_t := (\phi_{\pm})_t^* J_{\pm}$, and is generalized
K\"ahler with respect to these two complex structures.  These complex structures
$(J_{\pm})_t$ are moving via diffeomorphism, but when the evolution equation is
expressed with respect to the ``moving gauge,'' i.e. in terms of the pullback
data, one observes that they in fact satisfy a PDE of the form
\begin{align*}
 \dt J_{\pm} = \gD J + \Rm * J + DJ^{*2},
\end{align*}
where one can consult (\cite{STGK} Proposition 3.1) for the precise formula.

The presence of evolving complex structures makes this flow of generalized
K\"ahler structures perhaps a bit intimidating.  One purpose of Theorem
\ref{mainthm1} is to
show that if the initial generalized K\"ahler structure satisfies $[J_+,J_-] =
0$, then we in fact have that pluriclosed flow remains compatible with these
given $J_{\pm}$, and thus we do not need to consider the diffeomorphism modified
flows, and hence we can treat $J_{\pm}$ as fixed.

\subsection{Generalized K\"ahler manifolds with commuting complex structures}

In this subsection we recall some fundamental structural results concerning
generalized K\"ahler manifolds with commuting complex structures.  Most of our
discussion here is adopted from \cite{GaultApost}, which the reader can consult
for further information.

Let $(M^{2n}, g, J_{\pm})$ be a generalized K\"ahler manifold satisfying the
further condition
\begin{align} \label{Jcomm}
[J_+,J_-] = 0.
\end{align}
Let
\begin{align} \label{pidef}
\Pi := J_+ J_- \in \End(TM).
\end{align}
Using (\ref{Jcomm}), one easily shows that $\Pi^2 = \Id$.  It
follows that the eigenvalues of $\Pi$ are $\pm 1$, and the corresponding
eigenspace
decomposition yields a splitting of the tangent bundle
\begin{align*}
TM = T_+ M \oplus T_- M.
\end{align*}
This is the source of the terminology, ``generalized K\"ahler manifolds with
split tangent bundle.''  More precisely, one can define
\begin{align} \label{decomp}
 T_{\pm} M = \ker (\Pi \mp I) = \ker (J_+ \pm J_-).
\end{align}
We observe here an important general convention, which is that from this point
on many objects will be labeled with $\pm$, and this will \emph{almost always}
refer to the splitting above, and NOT to the usage of distinct complex
structures $J_{\pm}$.  In \S \ref{gkflow} we referred to the operators $d^c_+$,
which do in fact refer to $J_{\pm}$, and these will arise a few more times
below.  With this one exception, the notation $\pm$ signifies the decomposition
(\ref{decomp}), and in any case there is never any actual overlap of notation,
only a potential for confusion which it seems we are stuck with by existing
literature.

In particular, associated to (\ref{decomp}) one has an associated decomposition
$d = \gd_+ + \gd_-$
induced by $T^*M = T_+^* M \oplus T_-^*M$.  This splitting further refines the
decomposition of differential forms into types with respect to $J_+$), and
hence defines a decomposition
\begin{align*}
d = \gd_+ + \gd_- =  \del_+ + \delb_+ + \del_- + \delb_-
\end{align*}
Also, we can decompose the operator $\del\delb : C^{\infty} \to \Lambda^{1,1}$
into types, with the only two relevant pieces being
\begin{align*}
 \gd_{\pm} \gd^c_{\pm} = \pi_{\pm} \i \del \delb,
\end{align*}
where $\pi_{\pm}$ denotes the projections onto $\Lambda^{1,0}_{\pm} \wedge
\Lambda^{0,1}_{\pm}$, respectively.
Next we
record a theorem showing
that generalized K\"ahler metrics have a very
nice structure with respect to this splitting.
\begin{thm} (\cite{GaultApost} Theorem 4) \label{splitmetrics} Let $(M,g, J_+,
J_-)$ be a generalized
K\"ahler manifold with $[J_+, J_-] = 0$.  Let $\Pi = J_+ J_-$.  Then the $\pm
1$-eigenspaces of $\Pi$ are $g$-orthogonal $J_{\pm}$-holomorphic foliations on
whose leaves $g$ restricts to a K\"ahler metric. 
\end{thm}
Lastly, let us record a useful definition.
\begin{defn} \label{rankdef} Let $(M^{2n}, g, J_{\pm})$ be a generalized
K\"ahler manifold with
$[J_+,J_-] = 0$.  The \emph{rank} of $(J_+,J_-)$ is 
\begin{align*}
\rank_{J_{\pm}} := \dim_{\mathbb C}  T^{\mathbb C}_+ M.
\end{align*}
\end{defn}
\noindent Observe that if $\rank_{J_{\pm}} \in \{0,n\}$, then $J_+ = \pm J_-$,
and the
structure is K\"ahler.

\subsection{Adapted local coordinates}

In this subsection we describe local coordinates which are adapted to the local
product decomposition $TM = T_+ M \oplus T_- M$ described above.  This local
description of generalized K\"ahler structures with $[J_+,J_-] = 0$ was first
discovered in physical investigations into supersymmetry equations \cite{GHR},
where it is frequently referred to as ``bihermitian local product geometry.''   

\begin{lemma} \label{coordsplit} Let $(M^{2n}, g, J_{\pm})$ be a generalized
K\"ahler structure
satisfying $[J_+,J_-] = 0$.  Given $p \in M$ there exist local complex
coordinates adapted to the transverse foliations $T_{\pm}$.  That is, there is a
neighborhood $U = U_1 \times U_2 \subset \mathbb C^{n_1} \times \mathbb C^{n_2}$
such that $T_- U = T U_1, T_+ U = T U_2$.
\begin{proof} This follows from Theorem 2.1, see (\cite{GaultApost} \S 3). 
\end{proof}
\end{lemma}

As the complex structures preserve
the eigenspaces of $\Pi$, in the coordinates of Lemma \ref{coordsplit} the
complex
structures take the form
\begin{align} \label{Jaction}
J_{\pm} = \left( \begin{matrix}
               \pm J_m & \\
               0 & J_{n-m}
              \end{matrix} \right),
\end{align}
where
\begin{align*}
J_{1} := \left( \begin{matrix}
               0 & -1 \\
               1 & 0
              \end{matrix} \right),
\qquad J_{k} := \left(
\begin{matrix}
J_{1} & &\\
& \ddots & \\
& & J_{1}
\end{matrix} \right).
\end{align*}
These complex coordinates $\{z_i\}$ are thus only naturally oriented for the
complex structure $J_+$.  All
of our computations in local complex coordinates are on the complex manifold
$(M, J_+)$.
\begin{rmk} As noted above, when performing local coordinate calculations with
these coordinates, dummy
indices which are decorated with a $\pm$ will refer to only to summations over
vectors spanning $T^{\mathbb C}_{\pm} M$, respectively.  More generally, a
quantity like $g_{\ga_+ \bgb_+}$ means the metric evaluated at two vectors
$\frac{\del}{\del z^{\ga}}, \frac{\del}{\del \bar{z}^{\gb}} \in T^{\mathbb
C}_{\pm}M$.
\end{rmk}
With this decomposition in hand we can write a local coordinate expression for
the operators $\gd_{\pm} \gd^c_{\pm}$.  In particular, one has
\begin{gather}
\begin{split}
\gd_{\pm} \gd^c_{\pm} f =&\ f_{,\ga_{\pm} \bgb_{\pm}} \i dz_{\ga_{\pm}} \wedge d
\bar{z}_{\gb_{\pm}}.
\end{split}
\end{gather}
To finish, we observe how the structure of a generalized K\"ahler metric is
reflected in these coordinates.  In particular, since the decomposition $TM =
T_+ M \oplus T_- M$ is orthogonal for any generalized K\"ahler metric, it
follows that, locally,
\begin{align} \label{metricsplit}
g_{\ga_+ \bgb_-} = 0.
\end{align}
Also, the partial K\"ahler condition of Theorem \ref{splitmetrics} implies that
in these coordinates one has, locally,
 \begin{gather} \label{partialK}
  \begin{split}
  g_{\ga_+ \bgb_+,\mu_+} - g_{\mu_+ \bgb_+, \ga_+} =&\ g_{\ga_+ \bgb_+, \bmu_+}
- g_{\ga_+ \bmu_+,\bgb_+} = 0\\
     g_{\ga_- \bgb_-,\mu_-} - g_{\mu_- \bgb_-, \ga_-} =&\ g_{\ga_- \bgb_-,
\bmu_-}
- g_{\ga_- \bmu_-,\bgb_-} = 0\\
  \end{split}
 \end{gather}
Together (\ref{metricsplit}) and (\ref{partialK}) imply that the only
nonvanishing components of the torsion are
\begin{align} \label{torsionform}
T_{\ga_+ \gb_- \bgg_-} =&\ g_{\gb_- \bgg_-,\ga_+}, \qquad T_{\ga_- \gb_+ \bgg_+}
= g_{\gb_+ \bgg_+,\ga_-}
\end{align}
Finally, the pluriclosed condition combined with (\ref{metricsplit}) implies the
useful coordinate identity
\begin{align} \label{pccoord}
g_{\ga_+ \bgb_+,\mu_- \bnu_-} + g_{\mu_- \bnu_-,\ga_+ \bgb_+} = 0.
\end{align}

\section{Pluriclosed flow and commuting generalized K\"ahler geometry}
\label{pcfgkc}

In this section we establish that pluriclosed flow on generalized K\"ahler
manifolds with split tangent bundle preserves the generalized K\"ahler
condition, without the gauge fixing issue described in \S \ref{gkflow}.  The
first step is to establish a curvature identity in this setting which expresses 
$P$ in terms
of the first Chern classes of $T^{\mathbb C}_{\pm}
M$.  We can then use this to reduce the pluriclosed flow to a potential function
flow in \S \ref{reduction}, and motivate our Calabi-type conjecture in \S
\ref{ccsec}.

\subsection{A curvature identity}

To begin, note from Theorem \ref{splitmetrics} that $T^{\mathbb C}_{\pm} M$
are indeed holomorphic
vector bundles over $M$.  Moreover, as these are subbundles of $T^{\mathbb C}
M$, by restriction, any generalized K\"ahler metric $g$ induces Hermitian
metrics on these two bundles.  In the adapted local coordinates the local
coefficients of these two metrics are $g_{\ga_{\pm} \bgb_{\pm}}$.  These
metrics of course have associated curvature tensors, and in particular given a
generalized K\"ahler metric we will refer to the first Chern class
representatives for the induced bundle metrics as $\rho_{\pm}$.  By the well
known transgression formula for the first Chern class of a complex line bundle,
we have the local expression
\begin{align} \label{trans}
\left(\rho_{\pm}\right)_{i \bj} =&\ - \i \del_i \delb_j \log \det g_{\pm},
\end{align}
where we have added a factor of $2$ for notational convenience.  Furthermore,
exploiting the decomposition of $\Lambda^{1,1}_{\mathbb R}$ into
types via $T^{\mathbb C} M = T^{\mathbb C}_+ M \oplus T^{\mathbb C}_- M$, we can
consider the corresponding decomposition of $\rho_{\pm}$.  We will only have
need for the pieces of the form $\Lambda^{1,0}_{\pm} \wedge
\Lambda^{0,1}_{\pm}$, and we will denote the corresponding pieces of
$\rho_{\pm}$ decomposition by a $\pm$ superscript.  For concreteness sake, in
the adapted local coordinates one has the expressions
\begin{align*}
\left(\rho^{+}_+ \right)_{\ga_+ \bgb_+} =&\ - g^{\bgg_+ \gd_+} g_{\gd_+
\bgg_+,\ga_+ \bgb_+}  + g^{\bgg_+ \mu_+} g^{\bnu_+ \gd_+} g_{\mu_+
\bnu_+,\bgb_+} g_{\gd_+ \bgg_+,\ga_+}\\
\left(\rho^{-}_+ \right)_{\ga_- \bgb_-} =&\ - g^{\bgg_+ \gd_+} g_{\gd_+
\bgg_+,\ga_- \bgb_-}  + g^{\bgg_+ \mu_+} g^{\bnu_+ \gd_+} g_{\mu_+
\bnu_+,\bgb_-} g_{\gd_+ \bgg_+,\ga_-}\\
\left(\rho^{+}_- \right)_{\ga_+ \bgb_+} =&\ - g^{\bgg_- \gd_-} g_{\gd_-
\bgg_-,\ga_+ \bgb_+}  + g^{\bgg_- \mu_-} g^{\bnu_- \gd_-} g_{\mu_-
\bnu_-,\bgb_+} g_{\gd_- \bgg_-,\ga_+}\\
\left(\rho^{-}_- \right)_{\ga_- \bgb_-} =&\ - g^{\bgg_- \gd_-} g_{\gd_-
\bgg_-,\ga_- \bgb_-}  + g^{\bgg_- \mu_-} g^{\bnu_- \gd_-} g_{\mu_-
\bnu_-,\bgb_-} g_{\gd_- \bgg_-,\ga_-}
\end{align*}

With this background in place we give our calculation of $P$ in this special
setting.  To begin we recall a general coordinate formula for $P$.

\begin{lemma} (\cite{ST1}) Let $(M^{2n}, g, J)$ be a complex manifold with
pluriclosed metric $g$.  In local complex coordinates, one has
\begin{align*}
P(\gw)_{i \bj} =&\ - \i \left[ g^{\bq p} g_{i \bj, p\bq} + g^{\bq r} g^{\bs p}
\left( g_{r\bs,\bj} g_{p\bq,i} - g_{r\bs,\bj} g_{i\bq,p} - g_{r\bs,i}
g_{p\bj,\bq} \right) \right].
\end{align*}
\end{lemma}

\begin{prop} \label{curvid} Let $(M^{2n}, g, J_{\pm})$ be a generalized K\"ahler
manifold such
that $[J_+,J_-] = 0$.  Then
\begin{align*}
P(\gw) =&\ \rho_+^+ - \rho_-^+ - \rho_+^- + \rho_-^-.
\end{align*}
\begin{proof} We perform calculations in local coordinates as described above. 
We will frequently make use of (\ref{metricsplit}) - (\ref{pccoord}), sometimes
without explicitly saying so.  To begin we observe a simplification of the first
order terms in the local
expression for $P$.  In particular, since $g_{\ga_+ \bgb_-} = 0$, it follows
that
\begin{align*}
 g^{\bq r} g^{\bs p} &
\left( g_{r\bs,\bj} g_{p\bq,i} - g_{r\bs,\bj} g_{i\bq,p} - g_{r\bs,i}
g_{p\bj,\bq} \right)\\
=&\ g^{\bgb_+ \ga_+} g^{\bgg_+ \gd_{+}} \left( g_{\ga_+ \bgg_+,\bj} g_{\gd_+
\bgb_+,i} - g_{\ga_+ \bgg_+,\bj} g_{i \bgb_+,\gd_+} - g_{\ga_+ \bgg_+,i}
g_{\gd_+ \bj,\bgb_+} \right)\\
&\ + g^{\bgb_- \ga_-} g^{\bgg_- \gd_{-}} \left( g_{\ga_- \bgg_-,\bj} g_{\gd_-
\bgb_-,i} - g_{\ga_- \bgg_-,\bj} g_{i \bgb_-,\gd_-} - g_{\ga_- \bgg_-,i}
g_{\gd_- \bj,\bgb_-} \right).
\end{align*}
Using this identity we can compute
\begin{align*}
-P_{\mu_+ \bnu_+} =&\ \i \left[ g^{\bj i} g_{\mu_+ \bnu_+,i \bj} + g^{\bgb_+
\ga_+} g^{\bgg_+ \gd_{+}} \left( g_{\ga_+ \bgg_+,\bnu_+} g_{\gd_+
\bgb_+,\mu_+} - g_{\ga_+ \bgg_+,\bnu_+} g_{\mu_+ \bgb_+,\gd_+} - g_{\ga_+
\bgg_+,\mu_+}
g_{\gd_+ \bnu_+,\bgb_+} \right) \right. \\
&\ \qquad \qquad \left. + g^{\bgb_- \ga_-} g^{\bgg_- \gd_{-}} \left( g_{\ga_-
\bgg_-,\bnu_+} g_{\gd_-
\bgb_-,\mu_+} - g_{\ga_- \bgg_-,\bnu_+} g_{\mu_+ \bgb_-,\gd_-} - g_{\ga_-
\bgg_-,\mu_+}
g_{\gd_- \bnu_+,\bgb_-} \right) \right]\\
=&\ \i \left[ g^{\bgb_+ \ga_+} g_{\mu_+ \bnu_+,\ga_+ \bgb_+} + g^{\bgb_- \ga_-}
g_{\mu_+ \bnu_+,\ga_- \bgb_-} \right.\\
&\ \left. \qquad \qquad - g^{\bgb_+ \ga_+} g^{\bgg_+ \gd_+} g_{\ga_+
\bgg_+,\mu_+} g_{\gd_+ \bnu_+,\bgb_+} + g^{\bgb_- \ga_-} g^{\bgg_- \gd_-}
g_{\ga_- \bgg_-,\bnu_+} g_{\gd_-\bgb_-,\mu_+} \right]\\
=&\ \i \left[ g^{\bgb_+ \ga_+} g_{\ga_+ \bgb_+,\mu_+ \bnu_+} - g^{\bgb_- \ga_-}
g_{\ga_- \bgb_-,\mu_+ \bnu_+} \right.\\
&\ \left. \qquad \qquad - g^{\bgb_+ \ga_+} g^{\bgg_+ \gd_+} g_{\ga_+
\bgg_+,\mu_+} g_{\gd_+ \bgb_+,\bnu_+} + g^{\bgb_- \ga_-} g^{\bgg_- \gd_-}
g_{\ga_- \bgg_-,\bnu_+} g_{\gd_-\bgb_-,\mu_+} \right]\\
=&\ \left[ - \rho_+ + \rho_- \right]_{\mu_+ \bnu_+}.
\end{align*}
Observe that in the second to last line we applied the partial K\"ahler
condition (\ref{partialK}) for the term $g_{\mu_+ \bnu_+,\ga_+ \bgb_+}$ and the
pluriclosed condition (\ref{pccoord}) for the term $g_{\mu_+ \bnu_+,\ga_-
\bgb_-}$.
Next we compute
\begin{align*}
 -P_{\mu_+ \bnu_-} =&\ \i \left[ g^{\bj i} g_{\mu_+ \bnu_-,i\bj} + g^{\bgb_+
\ga_+} g^{\bgg_+ \gd_{+}} \left( g_{\ga_+ \bgg_+,\bnu_-} g_{\gd_+
\bgb_+,\mu_+} - g_{\ga_+ \bgg_+,\bnu_-} g_{\mu_+ \bgb_+,\gd_+} - g_{\ga_+
\bgg_+,\mu_+}
g_{\gd_+ \bnu_-,\bgb_+} \right) \right. \\
&\ \qquad \qquad \left. + g^{\bgb_- \ga_-} g^{\bgg_- \gd_{-}} \left( g_{\ga_-
\bgg_-,\bnu_-} g_{\gd_-
\bgb_-,\mu_+} - g_{\ga_- \bgg_-,\bnu_-} g_{\mu_+ \bgb_-,\gd_-} - g_{\ga_-
\bgg_-,\mu_+}
g_{\gd_- \bnu_-,\bgb_-} \right) \right]\\
=&\ \i \left[ g^{\bgb_+
\ga_+} g^{\bgg_+ \gd_{+}} \left( g_{\ga_+ \bgg_+,\bnu_-} g_{\gd_+
\bgb_+,\mu_+} - g_{\ga_+ \bgg_+,\bnu_-} g_{\mu_+ \bgb_+,\gd_+} \right) \right.\\
&\ \qquad \qquad \left. +
g^{\bgb_- \ga_-} g^{\bgg_- \gd_{-}} \left( g_{\ga_- \bgg_-,\bnu_-} g_{\gd_-
\bgb_-,\mu_+} - g_{\ga_-
\bgg_-,\mu_+}
g_{\gd_- \bnu_-,\bgb_-} \right) \right]\\
=&\ \i \left[ g^{\bgb_- \ga_-} g^{\bgg_- \gd_-} \left( g_{\ga_- \bgg_-,\bnu_-}
g_{\gd_- \bgb_-,\mu_+} - g_{\ga_- \bgg_-,\mu_+} g_{\gd_- \bgb_-,\bnu_-} \right)
\right]\\
=&\ \i \left[ g^{\bgb_- \ga_-} g^{\bgg_- \gd_-} \left( g_{\ga_- \bgg_-,\bnu_-}
g_{\gd_- \bgb_-,\mu_+} - g_{\gd_- \bgb_-,\mu_+} g_{\ga_- \bgg_-,\bnu_-} \right)
\right]\\
=&\ 0.
\end{align*}
It follows similarly that $P_{\mu_- \bar{\nu}_+} =0$.  Lastly we compute
\begin{align*}
 -P_{\mu_- \bnu_-} =&\ \i \left[ g^{\bj i} g_{\mu_- \bnu_-,i \bj} + g^{\bgb_+
\ga_+} g^{\bgg_+ \gd_{+}} \left( g_{\ga_+ \bgg_+,\bnu_-} g_{\gd_+
\bgb_+,\mu_-} - g_{\ga_+ \bgg_+,\bnu_-} g_{\mu_- \bgb_+,\gd_+} - g_{\ga_+
\bgg_+,\mu_-}
g_{\gd_+ \bnu_-,\bgb_+} \right) \right. \\
&\ \qquad \qquad \left. + g^{\bgb_- \ga_-} g^{\bgg_- \gd_{-}} \left( g_{\ga_-
\bgg_-,\bnu_-} g_{\gd_-
\bgb_-,\mu_-} - g_{\ga_- \bgg_-,\bnu_-} g_{\mu_- \bgb_-,\gd_-} - g_{\ga_-
\bgg_-,\mu_-}
g_{\gd_- \bnu_-,\bgb_-} \right) \right]\\
=&\ \i \left[ g^{\bgb_+ \ga_+} g_{\mu_- \bnu_-,\ga_+ \bgb_+} + g^{\bgb_- \ga_-}
g_{\mu_- \bnu_-,\ga_- \bgb_-} \right.\\
&\ \left. \qquad \qquad + g^{\bgb_+ \ga_+} g^{\bgg_+ \gd_+} g_{\ga_+
\bgg_+,\nu_-} g_{\gd_+ \bgb_+,\bnu_-} - g^{\bgb_- \ga_-} g^{\bgg_- \gd_-}
g_{\ga_- \bgg_-,\mu_-} g_{\gd_-\bnu_-,\bgb_-} \right]\\
=&\ \i \left[ - g^{\bgb_+ \ga_+} g_{\ga_+ \bgb_+,\mu_- \bnu_-} + g^{\bgb_-
\ga_-}
g_{\ga_- \bgb_-,\mu_- \bnu_-} \right.\\
&\ \left. \qquad \qquad + g^{\bgb_+ \ga_+} g^{\bgg_+ \gd_+} g_{\ga_+
\bgg_+,\nu_-} g_{\gd_+ \bgb_+,\bnu_-} - g^{\bgb_- \ga_-} g^{\bgg_- \gd_-}
g_{\ga_- \bgg_-,\mu_-} g_{\gd_-\bnu_-,\bgb_-} \right]\\
=&\ \left[ - \rho_- + \rho_+ \right]_{\mu_- \bnu_-}.
\end{align*}
\end{proof}
\end{prop}

Using this formula, we can easily observe that in this setting the tensor $P$
satisfies the conditions of a generalized K\"ahler metric, excepting positivity.
 One notes easily using the defining equations for $\gw_{\pm}$ that $\gw_- = -
\gw_+(\Pi \cdot, \cdot)$, therefore if we want to think of our flow solely in
terms of $\gw_+$ evolving by $P(\gw_+)$, we get an induced equation on $\gw_-$
flowing by $- P(\gw_+)(\Pi \cdot,\cdot)$.  The corollary below represents a
check that this will preserve the generalized K\"ahler condition.  The full
proof is in \S \ref{reduction}.

\begin{cor} \label{PGKcor} Let $(M^{2n}, g, J_{\pm})$ be a generalized K\"ahler
manifold satisfying $[J_+,J_-] = 0$.  Then
\begin{align} \label{PGKcor10}
d^c_+ P(\gw_+) = d^c_-\left[ P(\gw_+)( \Pi \cdot, \cdot) \right]
\end{align}
\begin{proof} First we compute, using that $d \rho_{\pm} = 0$,
\begin{align*}
d P = \left( \gd_+ + \gd_- \right) \left( \rho_+^+ - \rho_-^+ - \rho_+^- +
\rho_-^- \right) =&\ \gd_- \rho_+^+ - \gd_- \rho_-^+ - \gd_+ \rho_+^- + \gd_+
\rho_-^-
\end{align*}
Next we observe using the definition of $T_{\pm}$ that
\begin{align*}
P(\gw_+) (\Pi \cdot,\cdot) = \rho_+^+ - \rho_-^+ + \rho_+^- - \rho_-^-,
\end{align*}
and hence
\begin{align*}
d \left[ P(\Pi \cdot, \cdot) \right] =&\ \gd_- \rho_+^+ - \gd_- \rho_-^+ + \gd^+
\rho_+^- - \gd^+ \rho_-^-.
\end{align*}
Using (\ref{Jaction}), it follows that the actions of $J_{\pm}$ on $\gd_-
\rho_+^+ - \gd_- \rho_+^+$ are equal, whereas they differ by a sign on $\gd^+
\rho_+^- - \gd^+ \rho_-^-$.  The corollary follows.
\end{proof}
\end{cor}

\subsection{Reduction to a scalar potential} \label{reduction}

Our next goal is to develop a formal picture of the existence time for solutions
to the pluriclosed flow on generalized K\"ahler manifolds with $[J_+,J_-] = 0$. 
We first show that after the selection of background data the flow in this
setting can be reduced to a flow of a scalar potential function.  We then
exhibit a natural cohomology space with an associated ``positive cone" in
analogy
with the K\"ahler cone which suggests a conjectural maximal existence time for
the pluriclosed flow in this setting, specializing (\cite{ST2} Conjecture 5.2)

\begin{lemma} \label{reductionlemma} Let $(M^{2n}, g_0, J_{\pm})$ be a
generalized K\"ahler manifold. 
Let $h_{\pm}$ denote metrics on $T_{\pm} M$, respectively, and
let
\begin{align*}
 P = P(h_{\pm}) = \rho^+(h_+) - \rho^+(h_-) - \rho^-(h_+) + \rho^-(h_-).
\end{align*}
Furthermore, set $\gw_t = \gw_0 - t P$, and for a smooth function $f$ let
\begin{align*}
 \gw^f_t := \gw_t + \left( \gd_+ \gd_+^c - \gd_- \gd^c_- \right) f.
\end{align*}
Lastly, suppose $f_t$ solves
\begin{gather} \label{scalarflow}
\begin{split}
 \dt f =&\ \log \frac{\det g^f_{+} \det h_{-}}{\det h_{+} \det g^f_{-}},\\
 f(0) =&\ 0.
\end{split}
\end{gather}
Then $\gw^f_t$ is the unique solution to pluriclosed flow with initial condition
$\gw_0$.
\begin{proof} Using the given definitions, (\ref{trans}) and Proposition
\ref{curvid} we compute
\begin{align*}
 \dt \gw_t^f =&\ - \rho^+ (h_+) + \rho^+(h_-) + \rho^-(h_+) - \rho^-(h_-) + 
\left( \gd_+ \gd_+^c - \gd_- \gd_-^c \right) \log \frac{\det g_+^f \det
h_{-}}{\det h_{+}
\det g^f_-}\\
 =&\ \left[ \gd_+ \gd_+^c \left( \log \det h_+ - \log \det h_- + \log
\frac{\det g_+^f \det h_{-}}{\det h_{+}
\det g^f_-} \right) \right.\\
 &\ \qquad \left. + \gd_- \gd_-^c \left( \log \det h_- - \log \det h_+ - \log
\frac{\det g_+^f \det h_{-}}{\det h_{+}
\det g^f_-} \right) \right]\\
=&\ \left[ \gd_+ \gd_+^c \left( \log \det g^f_+ - \log \det g_-^f \right) -
\gd_-
\gd_-^c \left( \log \det g_+^f - \log \det g_-^f \right) \right]\\
=&\ - \rho_+^+ + \rho_-^+ + \rho_+^- - \rho_-^-\\
=&\ - P(\gw_t^f).
\end{align*}
The lemma follows.
\end{proof}
\end{lemma}

\begin{rmk} Equation (\ref{scalarflow}) is a fully nonlinear, parabolic equation
for $f$.  Crucially
however, it is \emph{nonconvex}, and therefore the well-understood theories of
Evans-Krylov \cite{Evans,Krylov} on $C^{2,\ga}$ regularity do not apply to this
equation, nor does the direct method of ``Calabi's third order estimate.'' 
Obtaining this estimate for this equation in general is a serious challenge,
which we
overcome in the case $n=2$ using delicate maximum principle arguments in \S
\ref{harnack}.
\end{rmk}

\begin{rmk} The corresponding fixed-point equation for (\ref{scalarflow}) has
appeared in
physics literature as a ``generalized Calabi-Yau" condition
\cite{Buscher,HullSSC,HullGCY,RocekMCY}.
\end{rmk}

\begin{proof} [Proof of Theorem \ref{mainthm1}] Since $\gw_0 > 0$ we can fix
some background metrics $h_{\pm}$ as in Lemma \ref{reductionlemma}, and some
$\ge > 0$ so that $\gw_0 - \ge P(h_{\pm}) > 0$.  From a calculation identical to
that in Lemma \ref{fdotev} one can show that equation (\ref{scalarflow}) is
strictly parabolic, and thus, for a smaller $\ge > 0$ there is a solution to
(\ref{scalarflow}) on $[0,\ge)$.  By Lemma \ref{reductionlemma} the family of
metrics metrics $\gw_t^f = \gw_0 - t P(h_{\pm}) + \i \left( \gd_+ \gd^c - \gd_-
\gd_-^c
\right) f$ is the unique solution to pluriclosed flow with initial condition
$\gw_0$.  Moreover, by Corollary \ref{PGKcor} these metrics are generalized
K\"ahler with
respect to $J_{\pm}$, finishing the proof.
 \end{proof}

\begin{rmk} In \cite{ST2} the author and Tian showed that in general the
pluriclosed flow can be reduced to a degenerate parabolic equation for a
$(0,1)$-form.  In this case, due to the extra complex structure, we have reduced
all the way to a scalar equation.  By using the formulas above, one can easily
recover that the $(0,1)$-form potential, in this case, is
\begin{align*}
\ga = \i \left(\delb_+ f - \delb_- f \right)
\end{align*}
\end{rmk}

\subsection{The positive cone} \label{conesec}

In this subsection we define a notion of a ``positive cone'' for generalized
K\"ahler metrics in analogy with the K\"ahler cone in K\"ahler geometry.

\begin{defn} Let $(M^{2n}, J_A, J_B)$ be a bihermitian manifold with $[J_A,J_B]
= 0$.  Given $\phi_A \in \Lambda^{1,1}_{J_A,\mathbb R}$, let $\phi_B = - \phi_A(
\Pi
\cdot, \cdot) \in \Lambda^{1,1}_{J_B,\mathbb R}$.  We say that $\phi_A$ is
\emph{formally generalized K\"ahler} if
\begin{gather} \label{FGK}
\begin{split}
d^c_{J_A} \phi_A =&\ - d^c_{J_B} \phi_B\\
d d^c_{J_A} \phi_A =&\ 0.
\end{split}
\end{gather}
\end{defn}

This definition is meant to codify the formal properties that a generalized
K\"ahler metric satisfies, other than positivity.  The content of Corollary
\ref{PGKcor} is that $P$ is formally generalized K\"ahler, and in particular the
path $\gw_0 - t P(h_{\pm})$ consists of formally generalized K\"ahler forms.  We
next study how long this path can possibly represent a class which contains
genuine generalized K\"ahler metrics.

\begin{defn} Let $(M^{2n}, g, J_A, J_B)$ denote a generalized K\"ahler manifold
such that $[J_A,J_B] = 0$.  Let
\begin{align*}
 \mathcal H := \frac{ \left\{ \phi_A \in \Lambda^{1,1}_{J_A, \mathbb R}\ |  \
\phi_A 
\mbox{ satisfies } (\ref{FGK}) \right\}}{ \left\{ \gd_+ \gd^c_+ f -
\gd_- \gd^c_- f \right\}}.
\end{align*}
Furthermore, set
\begin{align*}
 \mathcal P := \{ [\phi] \in \mathcal H \ | \ \exists \gw \in [\phi], \gw > 0
\}.
\end{align*}
We will refer to $\mathcal P$ as the \emph{positive cone}, and should be
interpreted in direct analogy with the K\"ahler cone on a K\"ahler manifold.
\end{defn}

\begin{defn} Let $(M^{2n}, g_0, J_{\pm})$ denote a generalized K\"ahler manifold
such that $[J_+,J_-] = 0$.  Let $h_{\pm}$ denote background metrics on
$T_{\pm}^{\mathbb C} M$, and let $P = P(h_{\pm})$ as above.  We define
\begin{align} \label{taustardef}
 \tau^* := \sup \{ t > 0 \ | \ [\gw_0- t P] \in \mathcal P \}.
\end{align}
\end{defn}
Our calculations above show that $\tau^*$ does not depend on the choice of
$h_{\pm}$.  Certainly $\tau^*$ represents the maximum possible existence time
for a solution
to pluriclosed flow in this setting.  In analogy with the Theorem of Tian-Zhang
(\cite{TianZhang}) for K\"ahler-Ricci flow, and as a specialization of the
corresponding general conjecture for pluriclosed flow (\cite{ST2}) we make the
conjecture that $\tau^*$ does indeed represent the smooth maximal existence
time.

\begin{conj} \label{coneconj} Let $(M^{2n}, g_0, J_{\pm})$ denote a generalized
K\"ahler manifold
with $[J_+,J_-] = 0$.  Let $\tau^*$ be defined as in (\ref{taustardef}).  Then
the solution to pluriclosed flow with initial condition $g_0$ exists on
$[0,\tau^*)$.
\end{conj}
\noindent Our Theorem \ref{mainthm2} verifies this conjecture in $n=2$, for most
cases.

\subsection{A Calabi-type conjecture} \label{ccsec}

Our discussion above shows that $P$ is a fully nonlinear elliptic operator in
the potential function $f$.  Moreover, the image of $P$ always represents the
same cohomology class in $\HH$, generalizing the usual first Chern class
associated to a complex manifold.  While pluriclosed flow is the natural
parabolic flow associated to this operator, it seems natural to also consider
the elliptic PDE problem suggested by this picture:

\begin{conj} \label{calconj} Let $(M^{2n}, g, J_{\pm})$ denote a generalized
K\"ahler structure
satisfying $[J_+,J_-] = 0$.  Given $\phi \in [P] \in \mathcal H$, there exists a
unique $\gw_{f} \in [\gw]$ such that $P(\gw_{f}) = \phi$.
\end{conj}

If a generalized K\"ahler structure satisfies $J_+ = \pm J_-$, then in fact
$(M,J_{\pm})$ are both K\"ahler manifolds, and generalized K\"ahler metrics on
$(M, J_{\pm})$ are the same as K\"ahler metrics on the individual complex
manifolds.  In particular, the whole discussion above reduces to the well-known
setting of K\"ahler-Ricci flow, and Conjecture \ref{calconj} reduces to the
famous Calabi conjecture \cite{Calabiconj, CC2}, resolved by Yau \cite{Yau}.

\section{Evolution Equations} \label{eveq}

In this section we record a number of evolution equations for some relevant
geometric quantities along a solution to pluriclosed flow in this context. 
Moreover we introduce the ``torsion potential" as remarked upon in the
introduction, and exhibit its remarkably simple evolution equation.  To begin
let us record some useful definitions and notational conventions.  First, we
will let $g$ denote a solution to pluriclosed flow as in Theorem \ref{mainthm1}.
 We will denote the induced Hermitian metrics on $T^{\mathbb C}_{\pm}$ as
$g_{\pm}$, respectively.  We use $h$ to denote an arbitrary fixed background
metric, with corresponding decomposition $h = h_+ + h_-$.  For lemmas involving
the potential function $f$ we assume that $\til{g}$ is a smooth one parameter
family of generalized K\"ahler metrics, and
\begin{align*}
\gw = \til{\gw} + \left( \gd_+ \gd^c_+ f - \gd_- \gd^c_- f \right),
\end{align*}
where $\til{\gw} = \gw_0 - t P(h)$ and $f$ evolves according to
(\ref{scalarflow}). 
The associated metric is denoted $g$.  The notation $\gD$ will refer to the
Chern Laplacian, acting via
\begin{align*}
\gD A = g^{\bj i} \N_i \N_{\bj} A,
\end{align*}
where $A$ is a section of some tensor bundle and $\N$ are the induced Chern
connections on these bundles.  Also, $\gG^{g_{\pm}}$ will refer to the local
coefficients of the Chern connection on $T^{\mathbb C}_{\pm}$ associated to the
metrics $g_{\pm}$, with $\gG^{h_{\pm}}$ defined analogously.  In this context we
let
\begin{align*}
\gU(g_{\pm},h_{\pm}) := \N^{g_{\pm}} - \N^{h_{\pm}}
\end{align*}
denote the tensor representing the difference of these two connections.  This
tensor arises in the evolution equations below, as do different ways to measure
its size.  In particular, as a section of $\Lambda^{1,0} \otimes \Lambda^{1,0}
\otimes T^{1,0}$, one can use different metrics to contract each
pair of indices.  We will denote these three choices as subscripts.  For
instance,
\begin{align*}
\brs{\gU(g_+,h_+)}^2_{g^{-1},g^{-1},h} = g^{\bj i} g^{\bgb_+ \ga_+} h_{\mu_+
\bnu_+} \gU_{i \ga_+}^{\mu_+} \bar{\gU}_{\bj \bgb_+}^{\bnu_+}.
\end{align*}
\begin{lemma} \label{halfdetev} Given the setup above,
\begin{gather}
 \begin{split}
  \dt \log \frac{ \det g_{\pm}}{\det h_{\pm}} =&\ \gD \log \frac{\det
g_{\pm}}{\det h_{\pm}} +
\frac{1}{2} \brs{T}^2 - \tr_g \rho(h_{\pm}).
 \end{split}
\end{gather}
\begin{proof} We give the proof for the case of $g_+$, the other case being
analogous.  To begin we compute
\begin{align*}
   \dt \log \frac{ \det g_+}{\det h_+} =&\ g^{\bgb_+ \ga_+} \dot{g}_{\ga_+
\bgb_+}\\
=&\ g^{\bgb_+ \ga_+} \left[ g^{\bq p} g_{\ga_+ \bgb_+,p\bq} + g^{\bq r} g^{\bs
p} \left( g_{r \bs,\bgb_+} g_{p\bq,\ga_+} - g_{r\bs,\bgb_+} g_{\ga_+ \bq,p} -
g_{r\bs,\ga_+} g_{p\bgb_+,\bq} \right) \right].
\end{align*}
Next we observe
\begin{align*}
 \gD \log \frac{\det g_+}{\det h_+} =&\ g^{\bq p} \left( \log \frac{\det
g_+}{\det h_+} \right)_{,\bq p}\\
 =&\ g^{\bq p} \left( g^{\bgb_+ \ga_+} g_{\ga_+ \bgb_+,\bq} - h^{\bgb_+ \ga_+}
h_{\ga_+ \bgb_+,\bq} \right)_{,p}\\
 =&\ g^{\bq p} \left( g^{\bgb_+ \ga_+} g_{\ga_+ \bgb_+,p\bq} - g^{\bgb_+ \gd_+}
g_{\gd_+ \bgg_+,p} g^{\bgg_+ \ga_+} g_{\ga_+\bgb_+,\bq} + \rho(h_+)_{p\bq}
\right).
\end{align*}
It follows that
\begin{align*}
 \dt \log \frac{\det g_+}{\det h_+} =&\ \gD \log \frac{\det g_+}{\det h_+} +
g^{\bgb_+ \ga_+} g^{\bq r} g^{\bs p} \left[ g_{r \bs,\bgb_+} g_{p\bq,\ga_+} -
g_{r\bs,\bgb_+} g_{\ga_+ \bq,p} -
g_{r\bs,\ga_+} g_{p\bgb_+,\bq} \right]\\
&\ + g^{\bq p} g^{\bgb_+ \gd_+} g^{\bgg_+ \ga_+} g_{\gd_+\bgg_+,p} g_{\ga_+
\bgb_+,\bq} - \tr_g \rho(h_+)\\
=&\ \gD \log \frac{\det g_+}{\det h_+} + \sum_{i=1}^5 A_i,
\end{align*}
where the terms $A_i$ are defined by the equality.  Next we observe that
\begin{align*}
 A_1 + A_4 =&\ g^{\bgb_+ \ga_+} g^{\bq r} g^{\bs p} g_{r \bs,\bgb_+}
g_{p\bq,\ga_+} + g^{\bq p} g^{\bgb_+ \gd_+} g^{\bgg_+ \ga_+} g_{\gd_+ \bgg_+,p}
g_{\ga_+ \bgb_+,\bq}\\
 =&\ g^{\bgb_+ \ga_+} g^{\bgg_+ \gd_+} g^{\bgs_+ \rho_+} g_{\gd_+ \bgs_+,\bgb_+}
g_{\rho_+ \bgg_+,\ga_+} + g^{\bgb_+ \ga_+} g^{\bgg_- \gd_-} g^{\bgs_- \rho_-}
g_{\gd_- \bgs_-,\bgb_+} g_{\rho_- \bgg_-,\ga_+}\\
&\ + g^{\bnu_+ \mu_+} g^{\bgb_+ \gd_+} g^{\bgg_+ \ga_+} g_{\gd_+ \bgg_+,\mu_+}
g_{\ga_+ \bgb_+,\bnu_+} + g^{\bnu_- \mu_-} g^{\bgb_+ \gd_+} g^{\bgg_+ \ga_+}
g_{\gd_+ \bgg_+,\mu_-} g_{\ga_+ \bgb_+,\bnu_-}\\
=&\ 2 g^{\bgb_+ \ga_+} g^{\bgg_+ \gd_+} g^{\bgs_+ \rho_+} g_{\gd_+
\bgs_+,\bgb_+} g_{\rho_+ \bgg_+,\ga_+} + g^{\bgb_+ \ga_+} g^{\bgg_- \gd_-}
g^{\bgs_- \rho_-} T_{\ga_+ \rho_- \bgg_-} T_{\bgb_+ \bgs_- \gd_-}\\
&\ + g^{\bgg_+ \ga_+} g^{\bgb_+ \gd_+} g^{\bnu_- \mu_-} T_{\mu_- \gd_+ \bgg_+}
T_{\bnu_- \ga_+ \bgb_+}\\
=&\ 2 g^{\bgb_+ \ga_+} g^{\bgg_+ \gd_+} g^{\bgs_+ \rho_+} g_{\gd_+
\bgs_+,\bgb_+} g_{\rho_+ \bgg_+,\ga_+} + \frac{1}{2} \brs{T}^2.
\end{align*}
Also we have, using (\ref{metricsplit}), (\ref{partialK}),
\begin{align*}
A_2 + A_3 =&\ - g^{\bgb_+ \ga_+} g^{\bq r} g^{\bs p} \left[ g_{r\bs,\bgb_+}
g_{\ga_+ \bq,p} + g_{r\bs,\ga_+} g_{p\bgb_+,\bq} \right]\\
=&\ - g^{\bgb_+ \ga_+} g^{\bgg_+ \gd_+} g^{\bgs_+ \rho_+} \left[ g_{\gd_+
\bgs_+,\bgb_+} g_{\ga_+ \bgg_+,\rho_+} + g_{\gd_+ \bgs_+,\ga_+} g_{\rho_+
\bgb_+,\bgg_+} \right]\\
=&\ - 2 g^{\bgb_+ \ga_+} g^{\bgg_+ \gd_+} g^{\bgs_+ \rho_+} g_{\rho_+
\bgg_+,\ga_+} g_{\gd_+ \bgs_+ \bgb_+}.
\end{align*}
Combining the above calculations yields the result.
\end{proof}
\end{lemma}

\begin{lemma} \label{fdotev} Given the setup above, one has
\begin{align*}
 \dt \dot{f} =&\ \gD \dot{f} - \tr_{g_+} P(h_{\pm}) + \tr_{g_-} P(h_{\pm}).
\end{align*}
\begin{proof}
 We directly compute using Lemma \ref{halfdetev}
\begin{align*}
 \dt \frac{\del f}{\del t} =&\ \dt \log \frac{\det g_+ \det h_{-}}{\det h_{+}
\det g_-}\\
=&\ \gD \log \frac{\det g_+}{\det h_+} + \frac{1}{2} \brs{T}^2 - \tr_g \rho(h_+)
- \gD \log \frac{\det g_-}{\det h_-} - \frac{1}{2} \brs{T}^2 + \tr_g \rho(h_-)\\
 =&\ \gD \dot{f} + \tr_g \left( \rho(h_-) - \rho(h_+) \right)\\
=&\ \gD \dot{f} + \tr_{g_+} \left( \rho_-^+ - \rho_+^+ \right) + \tr_{g_-}
\left( \rho_-^- - \rho_+^- \right)\\
=&\ \gD \dot{f} - \tr_{g_+} P(h_{\pm}) + \tr_{g_-} P(h_{\pm}).
\end{align*}
\end{proof}
\end{lemma}

\begin{lemma} \label{inversemetricev} Given the setup above,
\begin{gather*}
 \begin{split}
  \dt \tr_{g_{\pm}} h_{\pm} =&\ \gD \tr_{g_{\pm}} h_{\pm} -
\brs{\gU(g_{\pm},h_{\pm})}^2_{g^{-1},g^{-1},h} - \tr
g_{\pm}^{-1} h_{\pm} g^{-1}_{\pm} Q + g^{\bq p} g^{\bgb_{\pm} \ga_{\pm}} \left(
\Omega^{h_{\pm}}
\right)_{p \bq \ga_{\pm} \bgb_{\pm}}.
 \end{split}
\end{gather*}
\begin{proof} Once again we only deal with the case of $\tr_{g_+} h_+$.  To
begin we compute
\begin{align*}
 \dt \tr_{g_+} h_+ =&\ \dt \left( g^{\bgb_+ \ga_+} h_{\ga_+ \bgb_+} \right)\\
 =&\ - g^{\bgb_+ \gd_+} \dot{g}_{\gd_+ \bgg_+} g^{\bgg_+ \ga_+} h_{\ga_+
\bgb_+}\\
 =&\ - g^{\bgb_+ \gd_+} \left[ g^{\bq p} g_{\gd_+ \bgg_+,p\bq} + g^{\bq r}
g^{\bs p}
\left( g_{r\bs,\bgg_+} g_{p\bq,\gd_+} - g_{r\bs,\bgg_+} g_{\gd_+ \bq,p} -
g_{r\bs,\gd_+}
g_{p\bgg_+,\bq} \right) \right] g^{\bgg_+ \ga_+} h_{\ga_+ \bgb_+}.
\end{align*}
Next we observe
\begin{align*}
 \gD \tr_{g_+} h_+ =&\ g^{\bq p} \left( g^{\bgb_+ \ga_+} h_{\ga_+ \bgb_+}
\right)_{,\bq p}\\
 =&\ g^{\bq p} \left( - g^{\bgb_+ \gd_+} g_{\gd_+ \bgg_+,\bq} g^{\bgg_+ \ga_+}
h_{\ga_+ \bgb_+} + g^{\bgb_+ \ga_+} h_{\ga_+ \bgb_+,\bq} \right)_{,p}\\
 =&\ g^{\bq p} \left( g^{\bgb_+ \rho_+} g_{\rho_+ \bgs_+,p} g^{\bgs_+ \gd_+}
g_{\gd_+ \bgg_+,\bq} g^{\bgg_+ \ga_+} h_{\ga_+ \bgb_+} - g^{\bgb_+ \gd_+}
g_{\gd_+ \bgg_+,p\bq} g^{\bgg_+ \ga_+} h_{\ga_+ \bgb_+} \right.\\
 &\ \qquad + g^{\bgb_+ \gd_+} g_{\gd_+ \bgg_+,\bq} g^{\bgg_+ \rho_+} g_{\rho_+
\bgs_+,p} g^{\bgs_+ \ga_+} h_{\ga_+ \bgb_+} - g^{\bgb_+ \gd_+} g_{\gd_+
\bgg_+,\bq} g^{\bgg_+ \ga_+} h_{\ga_+ \bgb_+,p}\\
 &\ \left. \qquad - g^{\bgb_+ \gd_+} g_{\gd_+ \bgg_+,p} g^{\bgg_+ \ga_+}
h_{\ga_+ \bgb_+,\bq} + g^{\bgb_+ \ga_+} h_{\ga_+ \bgb_+,p\bq} \right).
\end{align*}
Combining these yields
\begin{align*}
\dt \tr_{g_+} h_+ =&\ \gD \tr_{g_+} h_+ - g^{\bgb_+ \gd_+} g^{\bq r} g^{\bs p}
g^{\bgg_+ \ga_+} h_{\ga_+ \bgb_+} \left[ g_{r\bs,\bgg_+} g_{p\bq,\gd_+} -
g_{r\bs,\bgg_+} g_{\gd_+ \bq,p} -
g_{r\bs,\gd_+}
g_{p\bgg_+,\bq} \right]\\
&\ - g^{\bq p} \left( g^{\bgb_+ \rho_+} g_{\rho_+ \bgs_+,p} g^{\bgs_+ \gd_+}
g_{\gd_+ \bgg_+,\bq} g^{\bgg_+ \ga_+} h_{\ga_+ \bgb_+}  \right.\\
 &\ \qquad + g^{\bgb_+ \gd_+} g_{\gd_+ \bgg_+,\bq} g^{\bgg_+ \rho_+} g_{\rho_+
\bgs_+,p} g^{\bgs_+ \ga_+} h_{\ga_+ \bgb_+} - g^{\bgb_+ \gd_+} g_{\gd_+
\bgg_+,\bq} g^{\bgg_+ \ga_+} h_{\ga_+ \bgb_+,p}\\
 &\ \left. \qquad - g^{\bgb_+ \gd_+} g_{\gd_+ \bgg_+,p} g^{\bgg_+ \ga_+}
h_{\ga_+ \bgb_+,\bq} + g^{\bgb_+ \ga_+} h_{\ga_+ \bgb_+,p\bq} \right)\\
=&\ \gD \tr_{g_+} h_+ + \sum_{i=1}^{7} A_i.
\end{align*}
We simplify some terms.  First, using (\ref{metricsplit}) and (\ref{partialK})
we have:
\begin{align*}
 A_1 =&\ - g^{\bgb_+ \gd_+} g^{\bq r} g^{\bs p}
g^{\bgg_+ \ga_+} h_{\ga_+ \bgb_+} g_{r\bs,\bgg_+} g_{p\bq,\gd_+}\\
=&\ - g^{\bgb_+ \gd_+} g^{\bnu_+ \mu_+} g^{\bgs_+ \rho_+}
g^{\bgg_+ \ga_+} h_{\ga_+ \bgb_+} g_{\mu_+ \bgs_+,\bgg_+} g_{\rho_+
\bnu_+,\gd_+} - g^{\bgb_+ \gd_+} g^{\bnu_- \mu_-} g^{\bgs_- \rho_-}
g^{\bgg_+ \ga_+} h_{\ga_+ \bgb_+} g_{\mu_- \bgs_-,\bgg_+} g_{\rho_-
\bnu_-,\gd_+}\\
=&\ - g^{\bgb_+ \gd_+} g^{\bnu_+ \mu_+} g^{\bgs_+ \rho_+}
g^{\bgg_+ \ga_+} h_{\ga_+ \bgb_+} g_{\mu_+ \bgg_+,\bgs_+} g_{\gd_+
\bnu_+,\rho_+} - g^{\bgb_+ \gd_+} g^{\bnu_- \mu_-} g^{\bgs_- \rho_-}
g^{\bgg_+ \ga_+} h_{\ga_+ \bgb_+} T_{\gd_+ \rho_- \bnu_-} T_{\bgg_+ \bgs_-
\mu_-}.
\end{align*}
Next
\begin{align*}
 A_2 + A_3 =&\ g^{\bgb_+ \gd_+} g^{\bq r} g^{\bs p}
g^{\bgg_+ \ga_+} h_{\ga_+ \bgb_+} \left[ g_{r\bs,\bgg_+} g_{\gd_+ \bq,p} +
g_{r\bs,\gd_+} g_{p\bgg_+,\bq} \right]\\
=&\ g^{\bgb_+ \gd_+} g^{\bnu_+ \mu_+} g^{\bgs_+ \rho_+}
g^{\bgg_+ \ga_+} h_{\ga_+ \bgb_+} \left[ g_{\mu_+ \bgs_+,\bgg_+} g_{\gd_+
\bnu_+,\rho_+} + g_{\mu_+ \bgs_+,\gd_+} g_{\rho_+ \bgg_+,\bnu_+} \right]\\
=&\ g^{\bgb_+ \gd_+} g^{\bnu_+ \mu_+} g^{\bgs_+ \rho_+}
g^{\bgg_+ \ga_+} h_{\ga_+ \bgb_+} \left[ g_{\mu_+ \bgg_+,\bgs_+} g_{\gd_+
\bnu_+,\rho_+} + g_{\gd_+ \bgs_+,\mu_+} g_{\rho_+ \bgg_+,\bnu_+} \right]\\
=&\ 2 g^{\bgb_+ \gd_+} g^{\bnu_+ \mu_+} g^{\bgs_+ \rho_+}
g^{\bgg_+ \ga_+} h_{\ga_+ \bgb_+} g_{\mu_+ \bgg_+,\bgs_+} g_{\gd_+
\bnu_+,\rho_+}.
\end{align*}
It follows that
\begin{align*}
 \sum_{i=1}^4 A_i =&\ g^{\bgb_+ \gd_+} g^{\bnu_+ \mu_+} g^{\bgs_+ \rho_+}
g^{\bgg_+ \ga_+} h_{\ga_+ \bgb_+} g_{\mu_+ \bgg_+,\bgs_+} g_{\gd_+
\bnu_+,\rho_+} - g^{\bgb_+ \gd_+} g^{\bnu_- \mu_-} g^{\bgs_- \rho_-}
g^{\bgg_+ \ga_+} h_{\ga_+ \bgb_+} T_{\gd_+ \rho_- \bnu_-} T_{\bgg_+ \bgs_-
\mu_-}\\
&\ - g^{\bq p} g^{\bgb_+ \rho_+} g^{\bgs_+ \gd_+} g^{\bgg_+ \ga_+} h_{\ga_+
\bgb_+} g_{\rho_+ \bgs_+,p} g_{\gd_+ \bgg_+,\bq}\\
=&\ g^{\bgb_+ \gd_+} g^{\bnu_+ \mu_+} g^{\bgs_+ \rho_+}
g^{\bgg_+ \ga_+} h_{\ga_+ \bgb_+} g_{\mu_+ \bgg_+,\bgs_+} g_{\gd_+
\bnu_+,\rho_+} - g^{\bgb_+ \gd_+} g^{\bnu_- \mu_-} g^{\bgs_- \rho_-}
g^{\bgg_+ \ga_+} h_{\ga_+ \bgb_+} T_{\gd_+ \rho_- \bnu_-} T_{\bgg_+ \bgs_-
\mu_-}\\
&\ - g^{\bnu_+ \mu_+} g^{\bgb_+ \rho_+} g^{\bgs_+ \gd_+} g^{\bgg_+ \ga_+}
h_{\ga_+ \bgb_+} g_{\rho_+ \bgs_+,\mu_+} g_{\gd_+ \bgg_+,\bnu_+} - g^{\bnu_-
\mu_-} g^{\bgb_+ \rho_+} g^{\bgs_+ \gd_+} g^{\bgg_+ \ga_+}
h_{\ga_+ \bgb_+} g_{\rho_+ \bgs_+,\mu_-} g_{\gd_+ \bgg_+,\bnu_-}\\
=&\ - g^{\bgb_+ \gd_+} g^{\bnu_- \mu_-} g^{\bgs_- \rho_-}
g^{\bgg_+ \ga_+} h_{\ga_+ \bgb_+} T_{\gd_+ \rho_- \bnu_-} T_{\bgg_+ \bgs_-
\mu_-} - g^{\bnu_-
\mu_-} g^{\bgb_+ \rho_+} g^{\bgs_+ \gd_+} g^{\bgg_+ \ga_+}
h_{\ga_+ \bgb_+} T_{\mu_- \rho_+ \bgs_+} T_{\bnu_- \bgg_+ \gd_+}\\
=&\ - \tr g_+^{-1} h_+ g_+^{-1} Q.
\end{align*}
Next we compute
\begin{align*}
 A_7 =&\ - g^{\bq p} g^{\bgb_+ \ga_+} h_{\ga_+ \bgb_+,p\bq}\\
 =&\ g^{\bq p} g^{\bgb_+ \ga_+} \left( \Omega^{h_+} \right)_{p \bq \ga_+
\bgb_+} - g^{\bq p} g^{\bgb_+ \ga_+} h^{\bnu_+ \mu_+} h_{\ga_+ \bnu_+,p}
h_{\mu_+ \bgb_+,\bq}.
\end{align*}
It follows that
\begin{align*}
 \sum_{i=5}^7 A_i =&\ - g^{\bq p} \left( g^{\bgb_+ \gd_+} g_{\gd_+ \bgg_+,\bq}
g^{\bgg_+ \rho_+} g_{\rho_+
\bgs_+,p} g^{\bgs_+ \ga_+} h_{\ga_+ \bgb_+} - g^{\bgb_+ \gd_+} g_{\gd_+
\bgg_+,\bq} g^{\bgg_+ \ga_+} h_{\ga_+ \bgb_+,p} \right.\\
 &\ \left. \qquad - g^{\bgb_+ \gd_+} g_{\gd_+ \bgg_+,p} g^{\bgg_+ \ga_+}
h_{\ga_+ \bgb_+,\bq} + g^{\bgb_+ \ga_+} h^{\bnu_+ \mu_+} h_{\ga_+ \bnu_+,p}
h_{\mu_+ \bgb_+,\bq} \right) + g^{\bq p} g^{\bgb_+ \ga_+} \left( \Omega^{h_+}
\right)_{p \bq \ga_+ \bgb_+}\\
=&\ - g^{\bq p} \left( \left( \gG^{g_+} \right)_{\bq \bgg_+}^{\bgb_+} \left(
\gG^{g_+} \right)_{p \rho_+}^{\ga_+} g^{\bgg_+ \rho_+} h_{\ga_+ \bgb_+} - \left(
\gG^{g_+} \right)_{\bq \bgg_+}^{\bgb_+} \left( \gG^{h_+} \right)_{p
\ga_+}^{\mu_+} g^{\bgg_+ \ga_+} h_{\mu_+ \bgb_+} \right.\\
&\ \qquad \left. - \left( \gG^{g_+} \right)_{p \gd_+}^{\ga_+} \left( \gG^{h_+}
\right)_{\bq \bgb_+}^{\bmu_+} g^{\bgb_+ \gd_+} h_{\ga_+ \bmu_+} + \left(
\gG^{h_+} \right)_{p\ga_+}^{\mu_+} \left( \gG^{h_+} \right)_{\bq
\bgb_+}^{\bgs_+} g^{\bgb_+ \ga_+} h_{\mu_+ \bgs_+} \right)\\
&\ \qquad + g^{\bq p} g^{\bgb_+ \ga_+} \left( \Omega^{h_+}
\right)_{p \bq \ga_+ \bgb_+}\\
=&\ - \brs{\gU(g_+,h_+)}^2_{g^{-1},g^{-1},h} + g^{\bq p} g^{\bgb_+ \ga_+} \left(
\Omega^{h_+}
\right)_{p \bq \ga_+ \bgb_+}.
\end{align*}
Collecting the above calculations yields the result.
\end{proof}
\end{lemma}

\begin{lemma} \label{partialmetricev} Given the setup above,
\begin{gather*}
 \begin{split}
  \dt \tr_{h_{\pm}} g_{\pm} =&\ \gD \tr_{h_+} g_{\pm} -
\brs{\gU(g_{\pm},h_{\pm})}_{g^{-1},h^{-1},g}^2 +
\tr_{h_{\pm}} Q - g^{\bj i} (h_{\pm}^{-1} g_{\pm} h_{\pm}^{-1})^{\bgg_{\pm}
\gd_{\pm}} \left(
\Omega^{h_{\pm}}
\right)_{i \bj \gd_{\pm} \bgg_{\pm}}.
 \end{split}
\end{gather*}
\begin{proof} Again we only give the proof for the case of $\tr_{h_+} g_+$, the
other case being analogous.  To begin we compute
\begin{align*}
 \dt \tr_{h_+} g_+ =&\ h^{\bgb_+ \ga_+} \left( P(h)_{\ga_+ \bgb_+} +
\dot{f}_{,\ga_+ \bgb_+} \right)\\
=&\ h^{\bgb_+ \ga_+} \left( \log \frac{ \det
g_+ \det h_-}{\det h_+ \det g_-} \right)_{,\ga_+ \bgb_+} + \tr_{h_+}
P(h)_+.
\end{align*}
Next we observe using (\ref{metricsplit}) - (\ref{pccoord}),
\begin{align*}
 \left( \log \frac{ \det g_+ \det h_-}{\det h_+ \det g_-}
\right)_{,\ga_+ \bgb_+} =&\ g^{\bnu_+ \mu_+} g_{\mu_+ \bnu_+,\ga_+\bgb_+} -
g^{\bnu_+ \mu_+} g^{\bgg_+ \gd_+} g_{\gd_+ \bnu_+,\ga_+} g_{\mu_+
\bgg_+,\bgb_+}\\
&\ - g^{\bnu_- \mu_-} g_{\mu_- \bnu_-,\ga_+ \bgb_+} + g^{\bnu_- \mu_-} g^{\bgg_-
\gd_-} g_{\gd_- \bnu_-,\ga_+} g_{\mu_- \bgg_-,\bgb_+}\\
&\ + h^{\bnu_- \mu_-} h_{\mu_- \bnu_-,\ga_+ \bgb_+} - h^{\bnu_- \mu_-} h^{\bgg_-
\gd_-} h_{\gd_- \bnu_-,\ga_+} h_{\mu_- \bgg_-,\bgb_+}\\
&\ - h^{\bnu_+ \mu_+} h_{\mu_+ \bnu_+,\ga_+\bgb_+} + h^{\bnu_+ \mu_+}
h^{\bgg_+ \gd_+} h_{\gd_+ \bnu_+,\ga_+} h_{\mu_+ \bgg_+,\bgb_+}\\
=&\ g^{\bnu_+ \mu_+} g_{\ga_+\bgb_+,\mu_+ \bnu_+} - g^{\bnu_+ \mu_+}
g^{\bgg_+ \gd_+} g_{\gd_+ \bnu_+,\ga_+} g_{\mu_+ \bgg_+,\bgb_+}\\
&\ + g^{\bnu_- \mu_-} g_{\ga_+ \bgb_+, \mu_- \bnu_-} + g^{\bnu_- \mu_-}
g^{\bgg_-
\gd_-} g_{\gd_- \bnu_-,\ga_+} g_{\mu_- \bgg_-,\bgb_+}\\
&\ + h^{\bnu_- \mu_-} h_{\mu_- \bnu_-,\ga_+,\bgb_+} - h^{\bnu_- \mu_-} h^{\bgg_-
\gd_-} h_{\gd_- \bnu_-,\ga_+} h_{\mu_- \bgg_-,\bgb_+}\\
&\ - h^{\bnu_+ \mu_+} h_{\mu_+ \bnu_+,\ga_+\bgb_+} + h^{\bnu_+ \mu_+}
h^{\bgg_+ \gd_+} h_{\gd_+ \bnu_+,\ga_+} h_{\mu_+ \bgg_+,\bgb_+}.
\end{align*}
Now we compute
\begin{align*}
\gD \tr_{h_+} g_+ =&\ g^{\bj i} \left( h^{\bgb_+ \ga_+} g_{\ga_+ \bgb_+}
\right)_{,i \bj}\\
=&\ g^{\bj i} \left( -
h^{\bgb_+ \gd_+} h_{\gd_+ \bgg_+,i} h^{\bgg_+ \ga_+} g_{\ga_+ \gb_+} + h^{\bgb_+
\ga_+} g_{\ga_+ \bgb_+,i} \right)_{,\bj}\\
=&\ g^{\bj i} \left[ h^{\bgb_+ \rho_+} h_{\rho_+ \bgs_+,\bj} h^{\bgs_+ \gd_+}
h_{\gd_+ \bgg_+,i} h^{\bgg_+ \ga_+} g_{\ga_+ \gb_+} - h^{\bgb_+ \gd_+} h_{\gd_+
\bgg_+,i\bj} h^{\bgg_+ \ga_+} g_{\ga_+ \gb_+} \right.\\
&\ + h^{\bgb_+ \gd_+} h_{\gd_+ \bgg_+,i} h^{\bgg_+ \rho_+} h_{\rho_+ \bgs_+,\bj}
h^{\bgs_+ \ga_+} g_{\ga_+ \gb_+} - h^{\bgb_+ \gd_+} h_{\gd_+ \bgg_+,i} h^{\bgg_+
\ga_+} g_{\ga_+ \bgb_+,\bj}\\
&\ \left. - h^{\bgb_+ \gd_+} h_{\gd_+ \bgg_+,\bj} h^{\bgg_+ \ga_+} g_{\ga_+
\bgb_+,i} + h^{\bgb_+ \ga_+} g_{\ga_+ \bgb_+,i\bj} \right].
\end{align*}
Combining the two above calculations yields
\begin{align*}
 h^{\bgb_+ \ga_+} &\ \left( \log \frac{ \det
g_+ \det h_-}{\det h_+ \det g_-} \right)_{,\ga_+ \bgb_+}\\
=&\ \gD \tr_{h_+} g_+ + h^{\bgb_+ \ga_+} \left[ - g^{\bnu_+ \mu_+}
g^{\bgg_+ \gd_+} g_{\gd_+ \bnu_+,\ga_+} g_{\mu_+ \bgg_+,\bgb_+} + g^{\bnu_-
\mu_-}
g^{\bgg_-
\gd_-} g_{\gd_- \bnu_-,\ga_+} g_{\mu_- \bgg_-,\bgb_+} \right.\\
&\ + h^{\bnu_- \mu_-} h_{\mu_- \bnu_-,\ga_+,\bgb_+} - h^{\bnu_- \mu_-} h^{\bgg_-
\gd_-} h_{\gd_- \bnu_-,\ga_+} h_{\mu_- \bgg_-,\bgb_+} - h^{\bnu_+ \mu_+}
h_{\mu_+ \bnu_+,\ga_+\bgb_+}\\
&\ \left. + h^{\bnu_+ \mu_+}
h^{\bgg_+ \gd_+} h_{\gd_+ \bnu_+,\ga_+} h_{\mu_+ \bgg_+,\bgb_+} \right]\\
&\ - g^{\bj i} \left[ h^{\bgb_+ \rho_+} h_{\rho_+ \bgs_+,\bj} h^{\bgs_+ \gd_+}
h_{\gd_+
\bgg_+,i} h^{\bgg_+ \ga_+} g_{\ga_+ \gb_+} - h^{\bgb_+ \gd_+} h_{\gd_+
\bgg_+,i\bj} h^{\bgg_+ \ga_+} g_{\ga_+ \bgb_+} \right.\\
&\ + h^{\bgb_+ \gd_+} h_{\gd_+ \bgg_+,i} h^{\bgg_+ \rho_+} h_{\rho_+ \bgs_+,\bj}
h^{\bgs_+ \ga_+} g_{\ga_+ \gb_+} - h^{\bgb_+ \gd_+} h_{\gd_+ \bgg_+,i} h^{\bgg_+
\ga_+} g_{\ga_+ \gb_+,\bj}\\
&\ \left. - h^{\bgb_+ \gd_+} h_{\gd_+ \bgg_+,\bj} h^{\bgg_+ \ga_+} g_{\ga_+
\bgb_+,i} \right]\\
=&\ \gD \tr_{h_+} g_+ + \sum_{i=1}^{11} A_i.
\end{align*}
Now we identify some terms.  First of all
\begin{align*}
A_3 + A_4 =&\ h^{\bgb_+ \ga_+} h^{\bnu_- \mu_-} \left[ h_{\mu_- \bnu_-,\ga_+
\bgb_+} - h^{\bgg_- \gd_-} h_{\gd_- \bnu_-,\ga_+} h_{\mu_- \bgg_-,\bgb_+}
\right]= - \tr_{h_+} \rho^+(h_-).
\end{align*}
Next we observe
\begin{align*}
 A_5 + A_6 =&\ h^{\bgb_+ \ga_+} h^{\bnu_+ \mu_+} \left[ - h_{\mu_+ \bnu_+,\ga_+
\bgb_+} + h^{\bgg_+ \gd_+} h_{\gd_+ \bnu_+,\ga_+} h_{\mu_+ \bgg_+,\bgb_+}
\right] = \tr_{h_+} \rho^+(h_+).
\end{align*}
Next
\begin{align*}
 A_8 + A_9 =&\ g^{\bj i} \left( h^{\bgb_+ \gd_+} h_{\gd_+ \bgg_+,i\bj} h^{\bgg_+
\ga_+} g_{\ga_+ \bgb_+} - h^{\bgb_+ \gd_+} h_{\gd_+ \bgg_+,i} h^{\bgg_+ \rho_+}
h_{\rho_+ \bgs_+,\bj} h^{\bgs_+ \ga_+} g_{\ga_+ \bgb_+} \right)\\
=&\ g^{\bj i} (h_+^{-1} g_+ h_+^{-1})^{\bgg_+ \gd_+} \left( h_{\gd_+
\bgg_+,i\bj} - h_{\gd_+ \bmu_+,i} h^{\bmu_+ \nu_+} h_{\nu_+ \bgg_+,\bj}
\right)\\
=&\ - g^{\bj i} (h_+^{-1} g_+ h_+^{-1})^{\bgg_+ \gd_+} \left( \Omega^{h_+}
\right)_{i \bj \gd_+ \bgg_+}.
\end{align*}
Also, using (\ref{metricsplit}) we have that
\begin{align*}
 A_2 =&\ h^{\bgb_+ \ga_+} g^{\bnu_- \mu_-} g^{\bgg_- \gd_-} g_{\gd_-
\bnu_-,\ga_+} g_{\mu_- \bgg_-,\bgb_+}\\
=&\ h^{\bgb_+ \ga_+} g^{\bnu_- \mu_-} g^{\bgg_- \gd_-} T_{\ga_+ \gd_- \bnu_-}
T_{\bgb_+ \bgg_- \mu_-}.
\end{align*}
We now consider the piece of $A_7 + A_{10} + A_{11}$ coming from the contraction
with the metric $g_-$.  In particular, we have
\begin{align*}
& \left( A_{7} + A_{10} + A_{11} \right)_-\\
=&\ g^{\bnu_- \mu_-} \left[ - h^{\bgb_+ \rho_+} h_{\rho_+ \bgs_+,\bnu_-}
h^{\bgs_+ \gd_+} h_{\gd_+ \bgg_+,\mu_-} h^{\bgg_+ \ga_+} g_{\ga_+ \bgb_+}
\right.\\
&\ \left. + h^{\bgb_+ \gd_+} h_{\gd_+ \bgg_+,\mu_-} h^{\bgg_+ \ga_+} g_{\ga_+
\bgb_+,\bnu_-} + h^{\bgb_+ \gd_+} h_{\gd_+ \bgg_+,\bnu_-}  h^{\bgg_+ \ga_+}
g_{\ga_+ \bgb_+,\mu_-} \right]\\
=&\ g^{\bnu_- \mu_-} \left[ - \left( \gG^{h_+} \right)_{\bnu_- \bgs_+}^{\bgb_+}
\left( \gG^{h_+} \right)_{\mu_- \gd_+}^{\ga_+} h^{\bgs_+ \gd_+} g_{\ga_+ \bgb_+}
+ \left( \gG^{g_+} \right)_{\bnu_- \bgb_+}^{\bgg_+} \left( \gG^{h_+}
\right)_{\mu_- \gd_+}^{\ga_+} h^{\bgb_+ \gd_+} g_{\ga_+ \bgg_+} \right.\\
&\ \qquad \qquad \left. + \left( \gG^{h_+} \right)_{\bnu_- \bgg_+}^{\bgb_+}
\left( \gG^{g_+} \right)_{\mu_- \ga_+}^{\rho_+} h^{\bgg_+ \ga_+} g_{\rho_+
\bgb_+} \right]\\
=&\ - g^{\bnu_- \mu_-} h^{\bgs_+ \gd_+} g_{\ga_+ \bgb_+} \left[ \left( \gG^{h_+}
- \gG^{g_+} \right)_{\mu_- \gd_+}^{\ga_+} \left( \gG^{h_+} - \gG^{g_+}
\right)_{\bnu_- \bgs_+}^{\bgb_+}  \right]\\
&\ + g^{\bnu_- \mu_-} h^{\bgs_+ \gd_+} g_{\ga_+ \bgb_+} g^{\bgg_+ \ga_+}
g_{\gd_+ \bgg_+,\mu_-} g^{\bgb_+ \rho_+} g_{\rho_+ \bgs_+,\bnu_-}\\
=&\ - g^{\bnu_- \mu_-} h^{\bgs_+ \gd_+} g_{\ga_+ \bgb_+} \left[ \left( \gG^{h_+}
- \gG^{g_+} \right)_{\mu_- \gd_+}^{\ga_+} \left( \gG^{h_+} - \gG^{g_+}
\right)_{\bnu_- \bgs_+}^{\bgb_+}  \right]\\
&\ + g^{\bnu_- \mu_-} h^{\bgs_+ \gd_+} g_{\gd_+ \bgb_+,\mu_-} g^{\bgb_+ \rho_+}
g_{\rho_+ \bgs_+,\bnu_-}\\
=&\ - g^{\bnu_- \mu_-} h^{\bgs_+ \gd_+} g_{\ga_+ \bgb_+} \left[ \left( \gG^{h_+}
- \gG^{g_+} \right)_{\mu_- \gd_+}^{\ga_+} \left( \gG^{h_+} - \gG^{g_+}
\right)_{\bnu_- \bgs_+}^{\bgb_+}  \right]\\
&\ + h^{\bgb_+ \ga_+} g^{\bnu_- \mu_-} g^{\bgg_+ \gd_+} T_{\mu_- \ga_+ \bgg_+}
T_{\bnu_- \bgb_+ \gd_+}.
\end{align*}
Lastly we observe
\begin{align*}
&\ A_1 + (A_7 + A_{10} + A_{11})_+\\
=&\ - h^{\bgb_+ \ga_+} g^{\bnu_+ \mu_+} g^{\bgg_+ \gd_+} g_{\gd_+ \bnu_+,\ga_+}
g_{\mu_+ \bgg_+, \bgb_+}\\
&\ + g^{\bnu_+ \mu_+} \left[ - h^{\bgb_+ \rho_+} h_{\rho_+ \bgs_+,\bnu_+}
h^{\bgs_+ \gd_+} h_{\gd_+ \bgg_+,\mu_+} h^{\bgg_+ \ga_+} g_{\ga_+ \bgb_+}
\right.\\
&\ \left. + h^{\bgb_+ \gd_+} h_{\gd_+ \bgg_+,\mu_+} h^{\bgg_+ \ga_+} g_{\ga_+
\bgb_+,\bnu_+} + h^{\bgb_+ \gd_+} h_{\gd_+ \bgg_+,\bnu_+}  h^{\bgg_+ \ga_+}
g_{\ga_+ \bgb_+,\mu_+} \right]\\
=&\ - h^{\bgb_+ \ga_+} g^{\bgg_+ \gd_+} g_{\mu_+ \bnu_+} \left( \gG^{g_+}
\right)_{\ga_+ \gd_+}^{\mu_+} \left( \gG^{g_+} \right)_{\bgb_+ \bgg_+}^{\bnu_+}
- h^{\bgs_+ \gd_+} g^{\bnu_+ \mu_+} g_{\ga_+ \bgb_+}\left( \gG^{h_+}
\right)_{\mu_+ \gd_+}^{\ga_+} \left( \gG^{h_+} \right)_{\bnu_+
\bgs_+}^{\bgb_+}\\
&\ + h^{\bgb_+ \gd_+} g^{\bnu_+ \mu_+} g_{\ga_+ \bgg_+} \left( \gG^{h_+}
\right)_{\mu_+ \gd_+}^{\ga_+} \left( \gG^{g_+} \right)_{\bnu_+ \bgb_+}^{\bgg_+}
+ h^{\bgg_+ \ga_+} g^{\bnu_+ \mu_+} g_{\gg_+ \bgb_+} \left( \gG^{g_+}
\right)_{\mu_+ \ga_+}^{\gg_+} \left( \gG^{h_+} \right)_{\bnu_+
\bgg_+}^{\bgb_+}\\
=&\ - g^{\bnu_+ \mu_+} h^{\bgs_+ \gd_+} g_{\ga_+ \bgb_+} \left[ \left( \gG^{h_+}
- \gG^{g_+} \right)^{\ga_+}_{\mu_+ \gd_+} \left( \gG^{h_+} - \gG^{g_+}
\right)_{\bnu_+ \bgs_+}^{\bgb_+} \right].
\end{align*}
Collecting up the above calculations yields the result.
\end{proof}
\end{lemma}

As we remarked in \S \ref{reduction}, the tensor playing the role of the
$1$-form version of pluriclosed flow (see \cite{ST2}) is
\begin{align*}
 \ga = \frac{\i}{2} \left( \delb^+ - \delb^- \right) f.
\end{align*}
It follows that the torsion potential takes the form
\begin{align*}
 \del \bar{\ga} = \left( \del^+ + \del^- \right) \frac{\i}{2} \left( \del^+ -
\del^- \right) f = \i \del^- \del^+ f.
\end{align*}
This tensor obeys a remarkable evolution equation.  While the term $\Psi$ in the
Lemma below looks
formidable, observe that it is at worst linear in the flowing connection, and
moreover completely vanishes against a K\"ahler background.

\begin{lemma} \label{torsionpotentialev} Given the above setup,
 \begin{align*}
  \dt \del_- \del_+ f =&\ \gD \del_- \del_+ f + \Psi,
 \end{align*}
 where
 \begin{gather} \label{psidef}
  \begin{split}
   \Psi_{\mu_+ \mu_-} :=&\ \left( \tr_{h_-} \N^h T^h \right)_{\mu_+ \mu_-} -
h^{\bgb_- \ga_-} h^{\bgs_+ \gg_+} T^h_{\mu_+ \mu_- \bgs_+} T^h_{\gg_+ \ga_-
\bgb_-}\\
  &\ - \left( \tr_{h_+} \N^h T^h \right)_{\mu_+ \mu_-} + h^{\bgb_+ \ga_+}
h^{\bgs_- \gg_-} T^h_{\mu_+ \mu_- \bgs_-} T^h_{\gg_- \ga_+ \bgb_+}\\
 &\ - \left( \tr_{g_-} \N^g \til{T} \right)_{\mu_+ \mu_-} + g^{\bgb_- \ga_-}
g^{\bgs_+ \gg_+} T_{\mu_- \mu_+ \bgs_+}
\til{T}_{\gg_+ \ga_- \bgb_-}\\
&\ + \left( \tr_{g_+} \N^g \til{T} \right)_{\mu_+ \mu_-} + g^{\bgb_- \ga_-}
g^{\bgs_+ \gg_+} T_{\mu_- \mu_+ \bgs_+}
\til{T}_{\gg_+ \ga_- \bgb_-}\\
&\ + g^{\bgb_- \ga_-} g^{\bgg_+ \gd_+} T_{\ga_- \mu_+ \bgg_+}
\til{T}_{\gd_+ \mu_- \bgb_-} - g^{\bgb_+ \ga_+} g^{\bgg_- \gd_-} T_{\ga_+ \mu_-
\bgg_-} \til{T}_{\gd_- \mu_+ \bgb_+}.
  \end{split}
 \end{gather}

 \begin{proof} To begin we compute
 \begin{align*}
\frac{\del}{\del t} f_{,\mu_+ \mu_-} =&\ \left( \log \frac{ \det g_+ \det
h_-}{\det h_+ \det g_-} \right)_{,\mu_+ \mu_-}\\
=&\ \left( g^{\bgb_+ \ga_+} g_{\ga_+ \bgb_+,\mu_+} - g^{\bgb_- \ga_-} g_{\ga_-
\bgb_-,\mu_+} + h^{\bgb_- \ga_-} h_{\ga_- \bgb_-,\mu_+} - h^{\bgb_+ \ga_+}
h_{\ga_+ \bgb_+,\mu_+} \right)_{,\mu_-}\\
=&\ g^{\bgb_+ \ga_+} g_{\ga_+ \bgb_+,\mu_+ \mu_-} - g^{\bgb_+ \gd_+} g^{\bgg_+
\ga_+} g_{\gd_+ \bgg_+,\mu_-} g_{\ga+ \bgb_+,\mu_+}\\
&\ - g^{\bgb_- \ga_-} g_{\ga_- \bgb_-,\mu_+ \mu_-} + g^{\bgb_- \gd_-} g^{\bgg_-
\ga_-} g_{\gd_-\bgg_-,\mu_-} g_{\ga_- \bgb_-,\mu_+}\\
&\ + h^{\bgb_- \ga_-} h_{\ga_- \bgb_-,\mu_+ \mu_-} - h^{\bgb_- \gd_-} h^{\bgg_-
\ga_-} h_{\gd_-\bgg_-,\mu_-} h_{\ga_- \bgb_-,\mu_+}\\
&\ - h^{\bgb_+ \ga_+} h_{\ga_+ \bgb_+,\mu_+ \mu_-} + h^{\bgb_+ \gd_+} h^{\bgg_+
\ga_+} h_{\gd_+ \bgg_+,\mu_-} h_{\ga+ \bgb_+,\mu_+}\\
=:&\ \sum_{i=1}^8 A_i.
 \end{align*}
 Let us simplify the terms coming from $h$ above.  First of all, we note that
 \begin{align*}
\left( \tr_{h_-} \N^h T^h \right)_{\mu_+ \mu_-} =&\ h^{\bgb_- \ga_-}
\N^h_{\mu_-} T^h_{\mu_+ \ga_- \bgb_-}\\
=&\ h^{\bgb_- \ga_-} \left( T^h_{\mu_+ \ga_- \bgb_-,\mu_-} - (\gG^h)_{\mu_-
\ga_-}^{\gg_-} T^h_{\mu_+ \gg_- \bgb_-} - (\gG^h)_{\mu_- \mu_+}^{\gg_+}
T^h_{\gg_+ \ga_- \bgb_-} \right)\\
=&\ h^{\bgb_- \ga_-} \left( h_{\ga_- \bgb_-,\mu_+ \mu_-} - h^{\bgs_- \gg_-}
h_{\ga_- \bgs_-,\mu_-} h_{\gg_- \bgb_-,\mu_+} - h^{\bgs_+ \gg_+} T^h_{\mu_-
\mu_+ \bgs_+} T^h_{\gg_+ \ga_- \bgb_-} \right)\\
=&\ A_5 + A_6 + h^{\bgb_- \ga_-} h^{\bgs_+ \gg_+} T^h_{\mu_+ \mu_- \bgs_+}
T^h_{\gg_+ \ga_- \bgb_-}
 \end{align*}
 A similar calculation yields
 \begin{align*}
- \left( \tr_{h_+} \N^h T^h \right)_{\mu_+ \mu_-} = A_7 + A_8 - h^{\bgb_+ \ga_+}
h^{\bgs_- \gg_-} T^h_{\mu_+ \mu_- \bgs_-} T^h_{\gg_- \ga_+ \bgb_+}.
 \end{align*}
 Next we observe that
 \begin{align*}
  \left[ \gD \del_- \del_+ f \right]_{\mu_+ \mu_-} =&\ g^{\bgb \ga} \left[ \N
\bar{\N} \del_- \del_+ f \right]_{\ga \bgb \mu_+ \mu_-}\\
  =&\ g^{\bgb \ga} \left[ \left( \del_- \del_+ f \right)_{\mu_- \mu_+, \bgb \ga}
- \gG_{\ga \mu_+}^{\gd_+} \left( \del_- \del_+ f \right)_{\mu_- \gd_+,\bgb} -
\gG_{\ga \mu_-}^{\gd_-} \left( \del_- \del_+ f \right)_{\gd_- \mu_+,\bgb}
\right]\\
=&\ g^{\bgb \ga} \left[ f_{,\mu_+ \mu_- \bgb \ga} - g^{\bgg_+ \gd_+} g_{\mu_+
\bgg_+,\ga} f_{,\gd_+ \mu_- \bgb} - g^{\bgg_- \gd_-} g_{\mu_- \bgg_-,\ga}
f_{,\mu_+ \gd_- \bgb} \right].
\end{align*}
First we express
\begin{align*}
 g^{\bgb \ga} f_{,\mu_+ \mu_- \bgb \ga} =&\ g^{\bgb_+ \ga_+} f_{,\mu_+ \mu_-
\bgb_+ \ga_+} + g^{\bgb_- \ga_-} f_{,\mu_+ \mu_- \bgb_- \ga_-}\\
 =&\ g^{\bgb_+ \ga_+} \left( g_{\ga_+ \bgb_+,\mu_+ \mu_-} - \til{g}_{\ga_+
\bgb_+,\mu_+ \mu_-} \right) + g^{\bgb_- \ga_-} \left( \til{g}_{\ga_-
\bgb_-,\mu_+ \mu_-} - g_{\ga_- \bgb_-,\mu_+ \mu_-} \right)\\
 =&\ A_1 + A_3 + g^{\bgb_- \ga_-} \til{g}_{\ga_- \bgb_-,\mu_+ \mu_-} - g^{\bgb_+
\ga_+} \til{g}_{\ga_+ \bgb_+,\mu_+ \mu_-}.
\end{align*}
Next we simplify
\begin{align*}
- g^{\bgb \ga} & \left[ g^{\bgg_+ \gd_+} g_{\mu_+ \bgg_+,\ga} f_{,\gd_+ \mu_-
\bgb} + g^{\bgg_- \gd_-} g_{\mu_- \bgg_-,\ga} f_{,\mu_+ \gd_- \bgb} \right]\\
=&\ - g^{\bgb_+ \ga_+} g^{\bgg_+ \gd_+} g_{\mu_+ \bgg_+,\ga_+} \left( g_{\gd_+
\bgb_+,\mu_-} - \til{g}_{\gd_+ \bgb_+,\mu_-} \right)\\
&\ + g^{\bgb_- \ga_-} g^{\bgg_+ \gd_+} g_{\mu_+ \bgg_+,\ga_-} \left(g_{\mu_-
\bgb_-,\gd_+} - \til{g}_{\mu_- \bgb_-,\gd_+} \right)\\
&\ - g^{\bgb_+ \ga_+} g^{\bgg_- \gd_-} g_{\mu_- \bgg_-,\ga_+} \left( g_{\mu_+
\bgb_+,\gd_-} - \til{g}_{\mu_+ \bgb_+,\gd_-} \right)\\
&\ + g^{\bgb_- \ga_-} g^{\bgg_- \gd_-} g_{\mu_- \bgg_-,\ga_-} \left( g_{\gd_-
\bgb_-,\mu_+} - \til{g}_{\gd_- \bgb_-,\mu_+} \right)\\
=&\ - g^{\bgb_+ \ga_+} g^{\bgg_+ \gd_+} g_{\ga_+ \bgg_+,\mu_+} g_{\gd_+
\bgb_+,\mu_-} + g^{\bgb_- \ga_-} g^{\bgg_- \gd_-} g_{\ga_- \bgg_-,\mu_-}
g_{\gd_- \bgb_-,\mu_+}\\
&\ + g^{\bgb_+ \ga_+} g^{\bgg_+ \gd_+} g_{\mu_+ \bgg_+,\ga_+} \til{g}_{\gd_+
\bgb_+,\mu_-} - g^{\bgb_- \ga_-} g^{\bgg_+ \gd_+} g_{\mu_+ \bgg_+,\ga_-}
\til{g}_{\mu_-\bgb_-,\gd_+}\\
&\ + g^{\bgb_+ \ga_+} g^{\bgg_- \gd_-} g_{\mu_- \bgg_-,\ga_+} \til{g}_{\mu_+
\bgb_+,\gd_-} - g^{\bgb_- \ga_-} g^{\bgg_- \gd_-} g_{\mu_- \bgg_-,\ga_-}
\til{g}_{\gd_- \bgb_-,\mu_+}\\
=&\ A_2 + A_4 + g^{\bgb_+ \ga_+} g^{\bgg_+ \gd_+} g_{\mu_+ \bgg_+,\ga_+}
\til{g}_{\gd_+
\bgb_+,\mu_-} - g^{\bgb_- \ga_-} g^{\bgg_+ \gd_+} g_{\mu_+ \bgg_+,\ga_-}
\til{g}_{\mu_-\bgb_-,\gd_+}\\
&\ + g^{\bgb_+ \ga_+} g^{\bgg_- \gd_-} g_{\mu_- \bgg_-,\ga_+} \til{g}_{\mu_+
\bgb_+,\gd_-} - g^{\bgb_- \ga_-} g^{\bgg_- \gd_-} g_{\mu_- \bgg_-,\ga_-}
\til{g}_{\gd_- \bgb_-,\mu_+},
\end{align*}
where in the penultimate line we observed a cancellation and applied
(\ref{partialK})
to two terms.  Combining our calculations above yields
\begin{align*}
 \dt \left( \del_- \del_+ f \right)_{\mu_+ \mu_-} =&\ \left[ \gD \del_- \del_+ f
+ \N^h T^h + T^h *_h T^h \right]_{\mu_+ \mu_-}\\
 &\ - g^{\bgb_- \ga_-} \til{g}_{\ga_- \bgb_-,\mu_+ \mu_-} + g^{\bgb_+
\ga_+} \til{g}_{\ga_+ \bgb_+,\mu_+ \mu_-}\\
&\ - g^{\bgb_+ \ga_+} g^{\bgg_+ \gd_+} g_{\mu_+ \bgg_+,\ga_+} \til{g}_{\gd_+
\bgb_+,\mu_-} + g^{\bgb_- \ga_-} g^{\bgg_+ \gd_+} g_{\mu_+ \bgg_+,\ga_-}
\til{g}_{\mu_-\bgb_-,\gd_+}\\
&\ - g^{\bgb_+ \ga_+} g^{\bgg_- \gd_-} g_{\mu_- \bgg_-,\ga_+} \til{g}_{\mu_+
\bgb_+,\gd_-} + g^{\bgb_- \ga_-} g^{\bgg_- \gd_-} g_{\mu_- \bgg_-,\ga_-}
\til{g}_{\gd_- \bgb_-,\mu_+}\\
=&\ \left[ \gD \del_- \del_+ f
+ \N^h T^h + T^h *_h T^h \right]_{\mu_+ \mu_-} + \sum_{i=1}^6 B_6.
\end{align*}
Lastly we simplify the $B_i$ terms.  First
\begin{align*}
- \left( \tr_{g_-} \N^g \til{T} \right)_{\mu_+ \mu_-} =&\ - g^{\bgb_- \ga_-}
\N^g_{\mu_-} \til{T}_{\mu_+ \ga_- \bgb_-}\\
=&\ g^{\bgb_- \ga_-} \left[ - \left( \til{T}_{\mu_+ \ga_- \bgb_-}
\right)_{,\mu_-} + (\gG^g)_{\mu_- \mu_+}^{\gg_+} \til{T}_{\gg_+ \ga_- \bgb_-} +
(\gG^g)_{\mu_- \ga_-}^{\gg_-} \til{T}_{\mu_+ \gg_- \bgb_-} \right]\\
=&\ g^{\bgb_- \ga_-} \left[ - \til{g}_{\ga_- \bgb_-,\mu_+ \mu_-} + g^{\bgs_+
\gg_+} g_{\mu_+ \bgs_+,\mu_-} \til{T}_{\gg_+ \ga_- \bgb_-} + g^{\bgs_- \gg_-}
g_{\ga_- \bgs_-,\mu_-} \til{g}_{\gg_-\bgb_-,\mu_+} \right]\\
=&\ B_1 + B_6 + g^{\bgb_- \ga_-} g^{\bgs_+ \gg_+} T_{\mu_- \mu_+ \bgs_+}
\til{T}_{\gg_+ \ga_- \bgb_-}
\end{align*}
Next
\begin{align*}
\left( \tr_{g_+} \N^g \til{T} \right)_{\mu_+ \mu_-} =&\ g^{\bgb_+ \ga_+}
\N^g_{\mu_+} \til{T}_{\mu_- \ga_+ \bgb_+}\\
=&\ g^{\bgb_+ \ga_+} \left[ \left( \til{T}_{\mu_- \ga_+ \bgb_+} \right)_{,\mu_+}
- (\gG^g)_{\mu_+ \mu_-}^{\gg_-} \til{T}_{\gg_- \ga_+ \bgb_-} - (\gG^g)_{\mu_+
\ga_+}^{\gg_+} \til{T}_{\mu_- \gg_+ \bgb_+} \right]\\
=&\ g^{\bgb_+ \ga_+} \left[ \til{g}_{\ga_+ \gb_+,\mu_+ \mu_-} - g^{\bgs_- \gg_-}
g_{\mu_- \bgs_-,\mu_+} \til{T}_{\gg_- \ga_+ \bgb_+} - g^{\bgs_+ \gg_+} g_{\ga_+
\bgs_+,\mu_+} \til{g}_{\gg_+ \bgb_+,\mu_-} \right]\\
=&\ B_2 + B_3 - g^{\bgb_- \ga_-} g^{\bgs_+ \gg_+} T_{\mu_- \mu_+ \bgs_+}
\til{T}_{\gg_+ \ga_- \bgb_-}.
\end{align*}
Lastly
\begin{align*}
 B_4 + B_5 =&\ g^{\bgb_- \ga_-} g^{\bgg_+ \gd_+} T_{\ga_- \mu_+ \bgg_+}
\til{T}_{\gd_+ \mu_- \bgb_-} - g^{\bgb_+ \ga_+} g^{\bgg_- \gd_-} T_{\ga_+ \mu_-
\bgg_-} \til{T}_{\gd_- \mu_+ \bgb_+}.
\end{align*}
The result follows.
\end{proof}
\end{lemma}

\begin{rmk} Observe that this evolution equation has the remarkable property
that the quadratic first order nonlinearity present throughout most evolution
equations associated to this flow has in this case completely disappeared.  This
will play a crucial role in obtaining estimates in this setting.  Indeed, this
property holds more generally for solutions to the pluriclosed flow, and will be
presented in future work.
\end{rmk}

\begin{lemma} \label{01formgenev} Let $(M^{2n}, \gw_t, J)$ be a solution to
pluriclosed flow, and
suppose $\gb_t \in \Lambda^{p,0}$ is a one-parameter family satisfying
\begin{align*}
 \dt \gb =&\ {\gD}_{g_t} \gb + \Psi,
\end{align*}
where $\Psi_t \in \Lambda^{p,0}$.  Then
\begin{align} \label{oneformflow}
 \dt \brs{\gb}^2 =&\ \gD \brs{\gb}^2 - \brs{\N \gb}^2 - \brs{\bar{\N} \gb}^2 -
p \IP{Q, \tr_g \left(\gb \otimes \bar{\gb} \right)} + 2 \Re \IP{\gb,\Psi}
\end{align}
\begin{proof} By direct computation we have
 \begin{align*}
  \dt \brs{\gb}^2 =&\ \dt g^{\bj_1 i_1} \dots g^{\bj_p i_p} {\gb}_{i_1\dots
i_p} \bar{\gb}_{\bj_1\dots \bj_p}\\
  =&\ - p g^{\bj_1 k} (-S + Q)_{k \bl} g^{\bl i_1} g^{\bj_2 i_2} \dots g^{\bj_p
i_p} {\gb}_{i_1 \dots i_p}
\bar{\gb}_{\bj_1 \dots \bj_p} + \IP{ \gD \gb, \bar{\gb}} +
\IP{\gb, \bar{\gD} \bar{\gb}} + 2 \Re \IP{\gb,\Psi}\\
=&\ p \IP{S - Q, \tr_g \left(\gb \otimes \bar{\gb} \right)} + \IP{ \gD \gb,
\bar{\gb}} +
\IP{\gb, \bar{\gD} \bar{\gb}} + 2 \Re \IP{\gb,\Psi}.
 \end{align*}
Next we observe the commutation formula
\begin{align*}
 \bar{\gD} \bar{\gb}_{\bj_1 \dots \bj_p} =&\ g^{\bk l} \N_{\bk} \N_l
\bar{\gb}_{\bj_1\dots\bj_p}\\
=&\ g^{\bk l} \N_l
\N_{\bk} \bar{\gb}_{\bj_1\dots\bj_p} - \sum_{r=1}^p g^{\bk l} \Omega_{\bk l
\bj_r}^{\bm} \bar{\gb}_{\bj_1\dots \bj_{r-1} \bm \bj_{r+1}\dots\bj_p}\\
=&\ \gD \bar{\gb}_{\bj_1\dots\bj_p} - \sum_{r=1}^p S_{\bj_r}^{\bm}
\bar{\gb}_{\bj_1\dots
\bj_{r-1} \bm \bj_{r+1}\dots\bj_p}.
\end{align*}
It follows that
\begin{align*}
 \IP{{\gb}, \bar{\gD} \bar{\gb}} =&\ g^{\bj_1 i_1} \dots g^{\bj_p i_p}
{\gb}_{i_1\dots i_p} \bar{\gD} \bar{\gb}_{\bj_1 \dots \bj_p}\\
 =&\ g^{\bj_1 i_1} \dots g^{\bj_p i_p} {\gb}_{i_1\dots i_p} \left[ \gD
\bar{\gb}_{\bj_1 \dots \bj_p} - \sum_{r=1}^p S_{\bj_r}^{\bm} \gb_{\bj_1\dots
\bj_{r-1}
\bm \bj_{r+1}\dots\bj_p} \right]\\
 =&\ \IP{ {\gb}, \gD \bar{\gb}} - p \IP{S, \gb \otimes \bar{\gb}}.
\end{align*}
Lastly observe the identity
\begin{align*}
 \gD \brs{\gb}^2 =&\ \IP{\gD \gb, \bar{\gb}} + \IP{{\gb}, \gD \bar{\gb}} +
\brs{\N
\gb}^2 + \brs{\bar{\N} \gb}^2.
\end{align*}
Combining the above calculations yields the lemma.
\end{proof}
\end{lemma}

\begin{rmk} This lemma yields a particularly clean estimate for $(p,0)$ forms
evolving by the heat equation against a wide class of Hermitian curvature flows,
as considered in \cite{STHCF}.  In the case of pluriclosed flow we have natural,
geometrically meaningful $(p,0)$ forms to which this observation can be
applied.  The crucial cancellation of the ``$S$'' term arising from the
variation
of the norm itself is similar to the very useful cancellation of the Ricci
curvature terms which occurs in the evolution of the gradient of a function
evolving by the heat flow against a Ricci-flow background.
\end{rmk}

\begin{lemma} \label{torsionpotentialnormev} Given the setup above, one has
\begin{align*}
 \dt \brs{\del_+ \del_- f}^2 =&\ \gD \brs{\del_+ \del_- f}^2 - \brs{\bar{\N}
\del_+ \del_- f}^2 - \brs{\N \del_+ \del_- f}^2 - 2 \IP{Q, \del_+ \del_- f
\otimes_g \bar{\del}_+ \delb_- f} + 2 \Re \IP{\del_+ \del_- f, \Psi},
\end{align*}
where $\Psi$ is as in (\ref{psidef}).
\begin{proof} This now follows directly from Lemmas \ref{torsionpotentialev} and
\ref{01formgenev}. 
\end{proof}
\end{lemma}

\section{A priori estimates} \label{estimates}

In this section and the next we obtain a number of a priori estimates which are
the main content of Theorem \ref{mainthm2}.   First in \S \ref{estsetup} we set
up the scalar reduction inside the positive cone.  Then we obtain an estimate
for the potential function in general dimensions in Lemma \ref{fbound}.  Then we
establish an estimate for $\dot{f}$ using structure special to dimension $n=2$. 
Note that the corresponding estimate in K\"ahler-Ricci flow corresponds to
uniform upper and lower bounds on the volume form, which for instance renders
upper and lower bounds on the metric equivalent.  In our setting it merely
controls the ratio of the volume forms of the metrics on the two bundles
$T^{\mathbb C}_{\pm} M$.  In \S \ref{mubsec} we establish an upper bound for the
metric in the presence of a lower bound.  This requires controlling potentially
troublesome torsion terms arising in the evolution of $\tr_h g$, which is where
the very helpful evolution equation of Lemma \ref{torsionpotentialnormev} plays
a 
crucial role.

\subsection{Setup} \label{estsetup}

Fix a time $\tau < \tau^*$, and fix arbitrary metrics $\til{h}_{\pm}$ on
$T_{\pm} M$ respectively, and observe
that by construction $\gw_0 - \tau P(\til{h}_{\pm}) \in \mathcal P_+$, and so
there exists $a \in C^{\infty}(M)$ such that
\begin{align*}
 \gw_0 - \tau P(\til{h}_{\pm}) + \left[ \gd_+ \gd^c_+ - \gd_- \gd^c_-
\right] a > 0.
\end{align*}
Now set $h_{\pm} = e^{\pm \frac{a}{2 \tau}} \til{h}_{\pm}$.  It follows that
\begin{align*}
 \gw_0 - \tau P(h_{\pm}) > 0.
\end{align*}
Moreover, by convexity one has a smooth one-parameter family of background
metrics
\begin{align*}
 \til{\gw}_t := \gw_0 - t P(h_{\pm}) > 0.
\end{align*}
With this choice of background data we let $f_t$ be the solution to
(\ref{scalarflow}) as in Lemma \ref{reduction}.  In the remainder of this
section we assume this setup, and all constants $C$ will depend on all of the
choices $n,\tau, h, \til{g}$.

\subsection{Estimate for the potential}

\begin{lemma} \label{fbound} Given the setup above, there exists a constant $C$
such that
\begin{align*}
 \brs{f} \leq C(1 + t).
\end{align*}
\begin{proof} Fix some constant $A > 0$ and let
\begin{align*}
 \Phi(x,t) := f(x,t) + t A.
\end{align*}
By direct computations we have
\begin{align*}
 \dt \Phi =&\ \log \frac{ \det g_+ \det h_{-}}{\det h_{+} \det g_-} + A\\
 =&\ \log \frac{\det g_+}{\det \til{g}_+} + \log \frac{\det \til{g}_{-}}{\det
g_-} + \log
\frac{ \det \til{g}_{+} \det h_{-}}{\det h_{+} \det \til{g}_-} + A\\
\geq&\ \log \frac{\det g_+}{\det \til{g}_+} + \log \frac{\det \til{g}_-}{\det
g_-},
\end{align*}
where the last inequality follows by choosing $A$ sufficiently large with
respect to the background data $\til{g}$, $h_{\pm}$.  At a minimum point for
$\Phi$,
it follows that $f$ is at a minimum, and so $\gd_{\pm} \gd^c_{\pm} f > 0$.  It
now follows from the maximum principle that the minimum of $\Phi$ is
nondecreasing, and so the lower bound for $\Phi$, and hence $f$, follows.  A
similar argument yields the upper bound.
\end{proof}
\end{lemma}

\subsection{Estimate for the time derivative of the potential}

\begin{prop} \label{fdotest} Given the setup above in the case $n=2, \rank=1$,
there
exists a
constant $C > 0$ so that
\begin{align*}
 \brs{\frac{\del f}{\del t}} \leq C.
\end{align*}
\begin{proof} 
To begin we observe the identity
\begin{align*}
 \gD f =&\ g^{\bga_+ \ga_+} f_{,\ga_+\bga_+} + g^{\bga_- \ga_-} f_{,\ga_-
\bga_-}\\
 =&\ g^{\bga_+ \ga_+} \left( g_{\ga_+ \bga_+} - \til{g}_{\ga_+ \bga_+} \right) +
g^{\bga_- \ga_-} \left( \til{g}_{\ga_- \bga_-} - g_{\ga_- \bga_-} \right)\\
 =&\ - g^{\bga_+ \ga_+} \til{g}_{\ga_+ \bga_+} + g^{\bga_- \ga_-} \til{g}_{\ga_-
\bga_-} .
\end{align*}
Now let
\begin{align*}
 \Phi(x,t) = (\tau - t) \dot{f} + f
\end{align*}
Combining the above observation with Lemma \ref{fdotev} yields the evolution
equation
\begin{align*}
 \dt \Phi =&\ \left(\tau - t \right) \ddot{f} - \dot{f} + \dot{f}\\
 =&\ (\tau - t) \left[ \gD \dot{f} - g^{\bga_+ \ga_+} P(h_{\pm})_{\ga_+ \bga_+}
+
g^{\bga_- \ga_-} P(h_{\pm})_{\ga_- \bga_-} \right]\\
=&\ \gD \Phi - \gD f + (\tau - t) \left[ - g^{\bga_+ \ga_+} P(h_{\pm})_{\ga_+
\bga_+}
+
g^{\bga_- \ga_-} P(h_{\pm})_{\ga_- \bga_-} \right]\\
=&\ \gD \Phi + g^{\bga_+ \ga_+} \left( \til{g}_t - (\tau - t) P(h_{\pm})
\right)_{\ga_+\bga_+} - g^{\bga_-\ga_-} \left( \til{g}_t - (\tau - t)
 P(h_{\pm})  \right)_{\ga_-\bga_-}\\
=&\ \gD \Phi + g^{\bga_+ \ga_+} (\til{g}_{\tau})_{\ga_+\bga_+} - g^{\bga_-\ga_-}
(\til{g}_{\tau})_{\ga_-\bga_-}.
\end{align*}
We will now apply the maximum principle to this identity.  Fix some constant $2A
> 0$.  Since $\brs{f}$ has a uniform bound by Lemma \ref{fbound}, if a maximum
of $\Phi$ is sufficiently large we can conclude that $\frac{\del f}{\del t} >
A$.  Using the evolution equation for $f$ this means that there is a uniform
constant $C$, so that if we fix coordinates such that $\til{h} = \Id$ at the
point in
consideration, one has
\begin{align*}
g_{\ga_+ \bga_+}> e^{\frac{A}{C}} g_{\ga_- \bga_-},
\end{align*}
which in turn implies that
\begin{align*}
 - g^{\bga_- \ga_-} \leq - e^{\frac{A}{C}} g^{\bga_+ \ga_+}
\end{align*}
In particular, if $A$ is sufficiently large with respect to the background
metric 
$\til{g}_T$ this implies that
\begin{align*}
g^{\bga_+ \ga_+} (\til{g}_{\tau})_{\ga_+\bga_+} - g^{\bga_-\ga_-}
(\til{g}_{\tau})_{\ga_-\bga_-} \leq&\ g^{\bga_+ \ga_+} \left(
(\til{g}_{\tau})_{\ga_+
\bga_+} - e^{\frac{A}{C}} (\til{g}_{\tau})_{\ga_- \bga_-} \right)\\
\leq&\ 0.
\end{align*}
By the maximum principle we conclude that a uniform upper bound for $\Phi$ holds
on any compact subinterval of $[0,T)$.  A similar line of reasoning yields the
lower bound for $\Phi$, finishing the proposition.
\end{proof}
\end{prop}

\subsection{Metric upper bound} \label{mubsec}

\begin{prop} \label{mub} Given the setup above, suppose there exists a constant
$K > 0$ so
that
\begin{align*}
g_t \geq \frac{1}{K} h.
\end{align*}
Then there exists a constant $C$ depending on $K$ and the background data so
that
\begin{align*}
\tr_h g + \brs{\del_+ \del_- f}^2 \leq C.
\end{align*}
\begin{proof} Fix a constant $A > 0$ yet to be
determined,
and
let
\begin{align*}
\Phi := \log \tr_h g + \tr_g h + A \brs{\del^+ \del^- f}^2.
\end{align*}
By combining Lemmas \ref{inversemetricev}, \ref{partialmetricev} and
\ref{torsionpotentialnormev} we obtain
\begin{gather} \label{mub5}
 \begin{split}
\left(\dt - \gD \right) \Phi \leq&\ \frac{1}{\tr_h g} \left[ \tr_h Q - g^{\bj i}
(h^{-1} g
h^{-1})^{\bgg \gd} \left(
\Omega^{h}
\right)_{i \bj \gd \bgg} \right]\\
&\ - \brs{\gU(g,h)}^2_{g^{-1},g^{-1},h} - \tr
g^{-1} h g^{-1} Q + g^{\bq p} g^{\bgb \ga} \left( \Omega^{h}
\right)_{p \bq \ga \bgb}\\
&\ + A \left[ - \brs{\bar{\N}
\del_+ \del_- f}^2 - \brs{\N \del_+ \del_- f}^2 - 2 \IP{Q, \del_+ \del_- f
\otimes_g \bar{\del}_+ \delb_- f} + 2 \Re \IP{\del_+ \del_- f, \Psi} \right].
 \end{split}
\end{gather}
It remains to establish an a priori upper bound for the right hand side.  To
facilitate estimates we choose complex coordinates at a point where $h_{i \bj} =
\gd_i^j$.  First we estimate
\begin{align} \label{mub10}
\frac{1}{\tr_h g} \tr_h Q \leq&\ \frac{C}{\tr_h g} \tr_h \left( \brs{T_g}_g^2 g
\right) \leq C \brs{T_g}_g^2.
\end{align}
Next we have
\begin{gather} \label{mub20}
\begin{split}
- \frac{1}{\tr_h g} g^{\bj i} (h^{-1} g
h^{-1})^{\bgg \gd} \left(
\Omega^{h}
\right)_{i \bj \gd \bgg} \leq&\ \frac{C K}{\tr_h g} h^{\bj i} \left( h^{-1} g
h^{-1} \right)^{\bgg \gd} \left( \Omega^h \right)_{i \bj \gd \bgg}\\
\leq&\ \frac{C K}{\tr_h g} \tr_h g \brs{\Omega^h}_h\\
\leq&\ C K.
\end{split}
\end{gather}
Next, since $Q \geq 0$ we have
\begin{align} \label{mub30}
 - \tr g^{-1} h g^{-1} Q \leq 0.
\end{align}
Also we estimate
\begin{align} \label{mub40}
 g^{\bq p} g^{\bgb \ga} \left( \Omega^{h}
\right)_{p \bq \ga \bgb} \leq C K^2 \brs{\Omega^h}_h \leq C K^2.
\end{align}
Next, using that $Q$ and $\del_+ \del_- f \otimes_g \delb_+ \delb_- f$ are
positive definite we obtain
\begin{align}
 - \brs{\N \del_+ \del_- f}^2 - 2 \IP{Q, \del_+ \del_- f
\otimes_g \bar{\del}_+ \delb_- f} \leq 0.
\end{align}
Plugging (\ref{mub10}) - (\ref{mub40}) into (\ref{mub5}) yields the preliminary
inequality
\begin{gather}
 \begin{split}
  \left( \dt - \gD \right) \Phi \leq&\ C \brs{T}^2 -
\brs{\gU(g,h)}^2_{g^{-1},g^{-1},h} - A
\brs{\bar{\N} \del_+ \del_- f}^2 + 2 A \Re \IP{\del_+ \del_- f, \mu} + C(K).
 \end{split}
\end{gather}
With this inequality one can more clearly see why the quantity $\Phi$ is chosen
as it is.  The main troublesome term is the $C \brs{T}^2$ arising from the
evolution of $\log \tr_h g$, the main quantity we want control over.  The
favorable $\brs{\bar{\N} \del_+ \del_- f}^2$ term will be used to control this,
which is why the term $\brs{\del_+ \del_- f}^2$ is introduced into $\Phi$.  On
the other hand this also introduces some ``junk'' terms arising from the
quantity $\Psi$, which involve the full Chern connection associated to $g$ (as
opposed to just the torsion).  We balance these out with the favorable
$\brs{\gU}^2$ term, which is the reason for the introduction of the term $\tr_g
h$ into $\Phi$.  We now make this precise.  We first observe the equality
\begin{align*}
 T_{\ga_- \gb_+ \bgg_+} =&\ g_{\gb_+ \bgg_+,\ga_-} = \til{g}_{\bg_+ \gg_+,\ga_-}
+ f_{,\gb_+ \ga_- \bgg_+} = \left( \til{T} + \bar{\N} \del_+ \del_- f
\right)_{\ga_- \gb_+ \bgg_+}.
\end{align*}
We conclude by the Cauchy-Schwarz inequality that
\begin{align} \label{mub50}
C \brs{T}^2 \leq&\ C \left( \brs{\til{T}}^2_g + \brs{\bar{\N} \del_+ \del_- f}^2
\right) \leq C(K) + C \brs{\bar{\N} \del_+ \del_- f}^2.
\end{align}
It remains to estimate the term $2 A \Re \IP{ \del_+ \del_- f, \Psi}$.  These
estimates are all similar.  For instance, we have
\begin{align*}
 g^{\bgs_+ \mu_+} & g^{\bgs_- \mu_-} g^{\bgb_- \ga_-} \N^g_{\mu_-}
\til{T}_{\mu_+ \ga_- \bgb_-} f_{,\bgs_- \bgs_+}\\
 =&\ g^{\bgs_+ \mu_+} g^{\bgs_- \mu_-} g^{\bgb_- \ga_-} \left[
\N^{\til{h}}_{\mu_-} \til{T}_{\mu_+ \ga_- \bgb_-} - \gU_{\mu_- \mu_+}^{\rho_+}
\til{T}_{\rho_+ \ga_- \bgb_-}  - \gU_{\mu_- \ga_-}^{\rho_-} \til{T}_{\mu_+
\rho_- \bgb_-} \right] f_{,\bgs_- \bgs_+} 
\end{align*}
Considering coordinates where $h_{i \bj} = \gd_{i}^j$ and $g_{i \bj} = \gl_i
\gd_i^j$ we have
\begin{align*}
 g^{\bgs_+ \mu_+} & g^{\bgs_- \mu_-} g^{\bgb_- \ga_-} \gU_{\mu_- \mu_+}^{\rho_+}
\til{T}_{\rho_+ \ga_- \bgb_-} f_{,\bgs_- \bgs_+}\\
\leq&\ g^{\bgs_+ \mu_+}
g^{\bgs_- \mu_-} \left[ \gd \gU_{\mu_- \mu_+}^{\rho_+} \gU_{\bgs_-
\bgs_+}^{\bar{\rho}_+} + \gd^{-1} \left( g^{\bgb_- \ga_-} \right)^2
\til{T}_{\rho_+ \ga_- \bgb_-} \til{T}_{\bar{\rho}_+ \bar{\ga}_- \gb_-}
f_{,\bgs_- \bgs_+} f_{,\mu_+ \mu_-} \right]\\
\leq&\ \gd \brs{\gU(g,h)}^2_{g^{-1},g^{-1},,h} + C(K,\gd^{-1},\til{T})
\brs{\del_+ \del_-
f}^2.
\end{align*}
By applying analogous estimates of this kind one arrives at the estimate, for
any choice of $\gd > 0$,
\begin{align} \label{mub60}
 2 A \IP{\del_+ \del_- f, \Psi} \leq&\ 2 A \gd
\brs{\gU(g,h)}^2_{g^{-1},g^{-1},h} +
C(\gd^{-1},K,\til{T}) \brs{\del_+ \del_- f}^2.
\end{align}
Plugging (\ref{mub50}) and (\ref{mub60}) into (\ref{mub40}) yields
\begin{align*}
 \left( \dt - \gD \right) \Phi \leq&\ \left(\gd A - 1 \right)
\brs{\gU(g,h)}^2_{g^{-1},g^{-1},h} + \left(C -  A \right)
\brs{\bar{\N} \del_+ \del_- f}^2 + C(K,\gd^{-1},\til{T}) \left(1 + \brs{\del_+
\del_- f}^2 \right).
\end{align*}
We choose $A$ large with respect to controlled constants and then $\gd =
\frac{1}{A}$, to arrive finally at the differential inequality
\begin{align*}
 \left( \dt - \gD \right) \Phi \leq&\ C(K,\gd^{-1},\til{T}) \left(1 +
\brs{\del_+ \del_- f}^2 \right) \leq C(K,\gd^{-1}, \til{T}) \left(1 + \Phi
\right).
\end{align*}
It follows from the maximum principle that there is a constant $C$ such that
\begin{align*}
\sup_{M \times [0,\tau)} \Phi \leq C.
\end{align*}
Since $\tr_g h + A \brs{\del^+ \del^- f}^2 > 0$, this immediately implies the
upper bound for $\tr_h g$.  Also, since $g_t$ is bounded below, we have that
$\log \tr_h g$ is bounded below, and so this implies the upper bound for
$\brs{\del^+ \del^- f}^2$.
\end{proof}
\end{prop}

\begin{rmk} This estimate in fact holds generally for solutions to pluriclosed
flow, where a more general ``torsion potential'' quantity plays the role of
$\del_+ \del_- f$.  This will be expounded upon in a future work.
\end{rmk}

\section{Harnack estimate} \label{harnack}

The purpose of this section is to establish $C^{\infty}$-regularity of a
solution to the pluriclosed flow in this setting assuming uniform equivalence of
the metric.  As stated in the introduction, the corresponding issue for
K\"ahler-Ricci flow can be resolved in at least two ways.  First, one can
try to directly apply the maximum principle to the norm of the Chern connection
potential, which is known as ``Calabi's third-order estimate'' \cite{Calabi}. 
Alternatively, apply the Evans-Krylov \cite{Evans,Krylov} theory. 
Both of
these methods rely on the ``convexity'' of the Monge-Ampere operator.  In
particular, the Evans-Krylov estimates apply only in the setting of a convex
operator, whilst in Calabi's approach the convexity appears in the form of a
favorable quadratic nonlinearity which arises in the relevant maximum principle
arguments.  Conversely, equation (\ref{scalarflow}) is \emph{nonconvex}, and so
neither of these approaches can succeed directly.  

In what follows we prove the relevant estimate in the case $n=2$.  The
restriction to dimension $n=2$ comes from Lemma \ref{harnacklemma}, where the
special structure of the torsion in this dimension is exploited to obtain a
favorable inequality for the evolution of the gradient of solutions to the heat
equation against a pluriclosed flow background.

\begin{thm} \label{EKregularity} Let $(M^4, g_0, J_{\pm})$ be a generalized
K\"ahler surface.  Let $g_t$ denote the solution to pluriclosed flow with
initial condition $g_0$.  Suppose the solution exists on $[0,T)$, and there is a
constant $A > 0$ so that
\begin{align*}
 A^{-1} g_0 \leq g_t \leq A g_0.
\end{align*}
Then given $k \in \mathbb N$ there
exists $C(A,g_0,k)$ such
that
\begin{align}
\sup_{M \times [0,T)} \brs{g}_{C^k} \leq C.
\end{align}
In particular, $g_T := \lim_{t \to T} g_t$ exists and is smooth, and the flow
extends smoothly past time $T$.
\end{thm}

\begin{rmk} The proof of Theorem \ref{EKregularity} proceeds in three steps. 
First we show that if the statement were false one can construct an ancient
solution to pluriclosed flow on $\mathbb C^2$ via a blowup argument, which has
the same uniform metric bounds, and is nonflat.  Next we show that, for such a
solution, any bounded ancient solution to the time-dependent Chern-Laplacian
heat equation is constant.  We then apply this rigidity to some special
geometric quantities which imply flatness of the underlying pluriclosed flow,
which is a contradiction which finishes the proof.
\end{rmk}

\subsection{Construction of blowup limits}

\begin{prop} \label{blowupcons} Let $(M^{2n}, g_t, J)$ be a solution to
pluriclosed flow on a
compact manifold $M$.  Suppose the solution exists on $[0,\tau)$
, and there is a constant $A > 0$ so that
\begin{align*}
 A^{-1} g_0 \leq g_t \leq A g_0.
\end{align*}
 Suppose furthermore that, setting $\gU := \N^C_{g_t} - \N^C_{g_0}$, one has
 \begin{align} \label{connblowup}
  \limsup_{t \to T} \brs{\gU}^2_{g_t} = \infty.
 \end{align}
There exists a blowup sequence of solutions which converges to a nonflat
solution of pluriclosed flow on $(-\infty,0] \times \mathbb C^n$ such that
\begin{align} \label{blowupbounds}
A^{-1} g_E \leq g_t \leq A g_E, \qquad \brs{g_t}_{C^k} \leq C(k,n,A).
\end{align}
where $g_E$ denotes the standard Euclidean metric on $\mathbb C^n$.
\begin{proof} Given a solution to pluriclosed flow as in the statement, we set
\begin{align*}
 \Phi(x,t) = \brs{\gU}^2 + \brs{\Omega} + \brs{\N T}.
\end{align*}
Assuming (\ref{connblowup}) holds, we choose a sequence of points
$\{(x_i,t_i)\}$ such that $t_i \to \tau$ and
\begin{align*}
 \Phi(x_i,t_i) = \sup_{M \times [0,t_i)} \Phi.
\end{align*}
Certainly $\lim_{i \to \infty} \Phi(x_i,t_i) = \infty$.  Since $M$ is compact
there exists $x_{\infty} \in M$ such that $\lim_{i \to \infty} x_i =
x_{\infty}$.  By fixing local complex coordinates around $x_{\infty}$ (mapping
$x_{\infty}$ to $0$) such that at $x_{\infty}$ one has $(g_0)_{i\bj}= \gd_i^j$. 
Now set $\gl_i = \Phi(x_i,t_i)^{-\frac{1}{2}}$, and in the fixed coordinates we
construct rescaled solutions $g_i(x,t) = g (x_i + \gl_i x, t_i + \gl_i^2 t)$. 
By construction one obtains that the metrics $g_i$ have uniform $C^2$ bounds on
$B_{\gl_i^{-1}}(0) \times [- \gl_i^{-2} t_i,0]$.  Moreover, by the smoothing
estimates of \cite{STHCF} Theorem 1.1 we obtain uniform estimates on all
derivatives of the Chern curvature and its torsion, which implies uniform $C^k$
bounds on the metric.  Also by construction, the rescaled metrics satisfy
$\Phi(0,0) = 1$.  One obtains a limit using the Arzela-Ascoli theorem which
satisfies the condition (\ref{blowupbounds}) by construction.
\end{proof}
\end{prop}

\subsection{Rigidity of ancient heat equation solutions}

The crucial vanishing result we need is a general result for ancient solutions
to the heat equation against a pluriclosed flow background which is uniformly
equivalent to the standard metric on $\mathbb C^2$.

\begin{lemma} \label{harnacklemma} Let $(M^{2n}, \gw_t, J)$ be a solution to
pluriclosed flow, and
suppose $f_t \in C^{\infty}(M)$ is a one-parameter family satisfying
\begin{align*}
 \dt f =&\ \gD_{g_t} f.
\end{align*}
Then
\begin{align} \label{gradfev}
 \dt \brs{\delb f}^2 =&\ \gD \brs{\delb f}^2 - \brs{\N \delb f}^2 -
\brs{\bar{\N} \delb f}^2 -
\IP{Q, \delb f \otimes \del f} + 2 \Re \IP{\delb f, T \circ \del \delb f}.
\end{align}
\begin{proof} To begin we observe an identity
\begin{align*}
\left(\delb \gD f\right)_{\bi} =&\ g^{\bk j} \N_{\bi} \N_j \N_{\bk} f\\
 =&\ g^{\bk j} \N_{\bi} \N_{\bk} \N_j f\\
 =&\ g^{\bk j} \left[ \N_{\bk} \N_{\bi} \N_j f + T_{\bk \bi}^{\bp} \N_{\bp} \N_j
f \right]\\
 =&\ g^{\bk j} \left[ \N_{\bk} \N_j \N_{\bi} f + T_{\bk \bi}^{\bp} \N_{\bp} \N_j
f \right]\\
 =&\ \bar{\gD} \N_{\bi} f + g^{\bk j} T_{\bk \bi}^{\bp} \N_{\bp} \N_j f\\
 =&\ \left( \bar{\gD} \delb f + T \circ \del \delb f \right)_{\bi}.
\end{align*}
Thus we may compute
\begin{align*}
 \dt \delb f =&\ \delb \gD f = \bar{\gD}_{g_t} \delb f + T \circ \del \delb f.
\end{align*}
The result now follows from Lemma \ref{01formgenev}.
\end{proof}
\end{lemma}

\begin{lemma} Let $(M^{4}, \gw_t, J)$ be a solution to pluriclosed flow, and
suppose $f_t \in C^{\infty}(M)$ is a one-parameter family satisfying
\begin{align*}
 \dt f =&\ \gD_{g_t} f.
\end{align*}
Then
\begin{align*}
 \dt \brs{\delb f}^2 \leq&\ \gD \brs{\delb f}^2 - \brs{\bar{\N} \delb f}^2.
\end{align*}
\begin{proof} We derive a special inequality for a $(0,1)$-form $\gb$ in this
context.  Since $n=2$, (\cite{ST1} Lemma 4.4) implies that
$Q = \frac{1}{2} \brs{T}^2 g$, thus $- \IP{Q, \gb \otimes \bar{\gb}} = -
\frac{1}{2} \brs{T}^2 \brs{\gb}^2$.  Also note that, in a unitary frame for
$g$, using the skew-symmetry of $T$ one has
\begin{align*}
 \brs{T}^2 =&\ \sum_{i,j,k=1}^2 T_{i j \bk} T_{\bi \bj k} = 2 (\brs{T_{1 2
\bar{1}}}^2 + \brs{T_{1 2 \bar{2}}}^2)
\end{align*}
Let us now analyze the quantity $\Re \IP{\gb,
T \circ \del \gb}$.  Working again in a unitary frame for $g$ we see
\begin{align*}
 \IP{\gb, T \circ \del \gb} =&\ \sum_{i,j,l = 1}^2 T^{}_{\bi \bj l}
\N_i \gb_{\bl} \gb_{j}\\
 =&\ \gb_1 \sum_{l=1}^2 T^{}_{\bar{2} \bar{1} l} \N_2 \gb_{\bl} + \gb_2
\sum_{l=1}^2 T^{}_{\bar{1} \bar{2} l} \N_1 \gb_{\bl}\\
=&\ \gb_1 T^{}_{\bar{2} \bar{1} 1} \N_2 \gb_{\bar{1}} + \gb_1
T^{}_{\bar{2} \bar{1} 2} \N_2 \gb_{\bar{2}} + \gb_2 T^{}_{\bar{1}\bar{2}
1} \N_1 \gb_{\bar{1}} + \gb_2 T^{}_{\bar{1}\bar{2} 2} \N_1 \gb_{\bar{2}}\\
\leq&\ \frac{1}{2} \left[ \brs{\gb_1}^2 \brs{T^{}_{\bar{2}\bar{1} 1}}^2 +
\brs{\N_2 \gb_{\bar{1}}}^2 + \brs{\gb_1}^2 \brs{T^{}_{\bar{2}\bar{1} 2}}^2 +
\brs{\N_2 \gb_{\bar{2}}}^2 \right.\\
&\ \left. \qquad + \brs{\gb_2}^2 \brs{T^{}_{\bar{1}\bar{2} 1}}^2 + \brs{\N_1
\gb_{\bar{1}}}^2 + \brs{\gb_2}^2 \brs{T^{}_{\bar{1}\bar{2} 2}}^2 + \brs{\N_1
\gb_{\bar{2}}}^2 \right]\\
=&\ \frac{1}{2} \left[ (\brs{\gb_1}^2 + \brs{\gb_2}^2) \left(
\brs{T^{}_{\bar{2} \bar{1} 1}}^2 + \brs{T^{}_{\bar{2} \bar{1} 2}}^2
\right) + \brs{\N \gb}^2 \right]\\
=&\ \frac{1}{2} \brs{\N \gb}^2 + \frac{1}{4} \brs{T}^2 \brs{\gb}^2.
\end{align*}
Collecting the above discussion it follows that
\begin{align*}
- \brs{\N \gb} + 2 \Re \IP{\gb, T^{} \circ \del \gb} - \IP{Q, \gb \otimes
\bar{\gb}} \leq 0.
\end{align*}
Using this inequality in (\ref{gradfev}) with $\gb = \delb f$ yields the result.
\end{proof}
\end{lemma}

\begin{prop} \label{heatrigid} Let $g_t$ be a solution to pluriclosed flow on
$\mathbb C^2 \times (-\infty,0]$ satisfying (\ref{blowupbounds}).  Suppose
$f(x,t)$ satisfies
\begin{align*}
 \dt f =&\ \gD f, \quad \sup_{(-\infty,0] \times \mathbb C^2} \brs{f} < \infty.
\end{align*}
Then $f$ is constant.
\begin{proof} Fix a constant $R > 1$, and let $\eta$ denote a cutoff function
for the ball
of radius $R > 0$.  Using the uniform bounds on the metric and connection from
(\ref{blowupbounds}) it
follows that
\begin{align*}
 \brs{\N \eta} + \brs{\N^2 \eta} \leq \frac{C}{R}.
\end{align*}
Now let
\begin{align*}
 \Phi(x,t) = \eta \left[ (t + \sqrt{R}) \brs{\delb f}^2 + f^2 \right]
\end{align*}
We directly compute
\begin{align*}
 \dt \Phi \leq&\ \eta \left[ \brs{\delb f}^2 + (t + \sqrt R) \gD \brs{\delb
f}^2
+ \gD
(f^2) - \brs{\delb f}^2 \right]\\
=&\ \gD \Phi - \gD \eta \left[ (t + \sqrt R) \brs{\delb f}^2 + f^2 \right]
-
\IP{\N
\eta, \N  \left[ (t + \sqrt R) \brs{\delb f}^2 + f^2 \right]}\\
\leq&\ \gD \Phi + \frac{ C\brs{t + \sqrt R}}{R}
\end{align*}
Applying the maximum principle to this inequality on $[-\sqrt R,0]$ yields the
estimate
\begin{align*}
 \sup_{\mathbb C^2} \Phi(x,0) \leq&\ \sup_{\mathbb C^2} \Phi(x,-\sqrt R) + C
\leq
\sup_{\mathbb C^2} f^2 + C \leq A + C.
\end{align*}
It follows that, for all $R$,
\begin{align*}
\sqrt{R} \brs{\delb f}^2(0,0) \leq&\ \Phi(0,0) \leq C
\end{align*}
It follows that $\brs{\delb f}^2(0,0) = 0$.  But this estimate can be
repeated
using cutoff functions centered at an arbitrary point in $\mathbb C^2 \times
(-\infty,0]$, and so $\delb f \equiv 0$ identically.  Thus at each time slice
$f$ is constant, and then from the evolution equation for $f$ it follows
that
$\dt f \equiv 0$, and so $f$ is constant on $\mathbb C^2 \times
(-\infty,0]$.
\end{proof}
\end{prop}

\subsection{Rigidity of blowup limits}

In this subsection we establish a rigidity theorem for ancient solutions in the
case $n=2$.  One key input is a further identity special to the this dimension.

\begin{lemma} \label{n2traceev} When $n=2$ and $h_+$ is a flat metric, one has
\begin{align*}
\dt \tr_{h_+} g_+ =&\ \gD \tr_{h_+} g_+ - \IP{\del^+ \log \frac{\det g_+ \det
h_-}{\det h_+ \det g_-}, \delb^+ \log \frac{\det g_+ \det g_-}{\det h_+ \det
h_-}}_h.
\end{align*}
\begin{proof} Beginning with the result of Lemma \ref{partialmetricev}, we first
observe that since $h_+$ is flat the last term involving $\Omega^{h_+}$ will
vanish.  Using that $n=2$ we proceed to simplify the first order terms.  First
of all,
\begin{align*}
- \brs{\gU(g_+,h_+)}^2_{g,h} =&\ - h^{\bgb_+ \ga_+} g^{\bgg_+ \gd_+} g_{\mu_+
\bnu_+} \gG^{\mu_+}_{\ga_+ \gd_+} \gG_{\bgb_+ \bgg_+}^{\bnu_+} - h^{\bgb_+
\ga_+} g^{\bgg_- \gd_-} g_{\mu_+ \bnu_+} \gG^{\mu_+}_{\gd_- \ga_+}
\gG^{\bnu_+}_{\bgg_- \bgb_+}\\
=&\ - h^{\bga_+ \ga_+} g^{\bga_+ \ga_+} g^{\bga_+ \ga_+} g_{\ga_+ \bga_+,\ga_+}
g_{\ga_+ \bga_+,\bga_+} - h^{\bga_+ \ga_+} g^{\bga_- \ga_-} g^{\bga_+ \ga_+}
g_{\ga_+ \bga_+,\ga_-} g_{\ga_+ \bga_+,\bga_-}
\end{align*}
Also
\begin{align*}
\tr_{h_+} Q =&\ h^{\bgb_+ \ga_+} Q_{\ga_+ \bgb_+}\\
=&\ h^{\bga_+ \ga_+} \left( g^{\bga_- \ga_-} g^{\bga_+ \ga_+} T_{\ga_+ \ga_-
\bga_+} T_{\bga_+ \bga_- \ga_+} + g^{\bga_- \ga_-} g^{\bga_- \ga_-} T_{\ga_+
\ga_- \bga_-} T_{\bga_+ \bga_- \ga_-} \right)\\
=&\ h^{\bga_+ \ga_+} g^{\bga_- \ga_-} g^{\bga_+ \ga_+} g_{\ga_+ \bga_+,\ga_-}
g_{\ga_+ \bga_+,\bga_-} + h^{\bga_+ \ga_+} g^{\bga_- \ga_-}g^{\bga_- \ga_-}
g_{\ga_- \bga_-,\ga_+} g_{\ga_- \bga_-,\bga_+}.
\end{align*}
Combining the above two expressions yields the inner product claimed above.
\end{proof}
\end{lemma}

\begin{thm} \label{rigiditythm} Let $g_t$ be a solution to pluriclosed flow on
$\mathbb C^2 \times (-\infty,0]$ satisfying (\ref{blowupbounds}).  Then $g_t
\equiv g_0$ is a flat metric.
\begin{proof} Since we are working on $\mathbb C^2$ against a flat background
metric, by Lemma \ref{fdotev} we see that $\frac{\del f}{\del t}$ satisfies the
time-dependent heat equation.  Moreover, using the uniform equivalence of $g$
with $g_{E}$ certainly $\brs{\frac{\del f}{\del t}} \leq C$.  It follows from
Proposition \ref{heatrigid} that $\frac{\del f}{\del t} \equiv c$.  Using this,
one 
observes from Lemma \ref{n2traceev} that $\tr_{h_{\pm}} g_{\pm}$ satisfy the
heat
equation, and are moreover bounded.  Thus by Proposition \ref{heatrigid} we
obtain that $\tr_{h_{\pm}} g_{\pm}$ are constant functions, and hence
$g$ is flat.
\end{proof}
\end{thm}

\subsection{Proof of Theorem \ref{EKregularity}}

\begin{proof}[Proof of Theorem \ref{EKregularity}] Let $\gU = \N^C_{g_t} -
\N^C_{g_0}$.  First we show an estimate for $\gU$.  Assuming that $\brs{\gU}^2$
blows up at time $\tau$, by Proposition \ref{blowupcons} we obtain a nonflat
blowup solution to pluriclosed flow on $(-\infty,0] \times \mathbb C^2$
satisfying (\ref{blowupbounds}).  However, by Theorem \ref{rigiditythm} this
solution is indeed flat, which is a contradiction.  A similar blowup argument
can be used to establish all of the higher $C^k$ estimates as well.
\end{proof}

\section{Proof of Theorem \ref{mainthm2}} \label{mainsec}

In this section we establish Theorem \ref{mainthm2}.  Our proof consists of
putting together the a priori estimates of \S \ref{estimates} to prove the main
analytic result, Theorem \ref{halfample}.  Then we verify case by case that the
manifolds stated in Theorem \ref{mainthm2} satisfy the hypotheses of Theorem
\ref{halfample}, finishing the proof of long time existence.  We end with some
remarks on the convergence behavior.

\begin{thm} \label{halfample} Let $(M^4, g, J_{\pm})$ be a compact generalized
K\"ahler surface
satisfying $[J_+,J_-] = 0$.  Suppose further that $c_1(T^{\mathbb C}_{+}) \leq
0$ or $c_1(T^{\mathbb C}_-) \leq 0$.  Then Conjecture
\ref{coneconj} holds for $(M^4, g_{\pm})$.  In particular, given $g_0$ a
generalized K\"ahler metric on $M$, and defining $\tau^*$ via
(\ref{taustardef}), the solution to pluriclosed flow with initial condition
$g_0$ exists smoothly on $[0,\tau^*)$.
 \begin{proof} Suppose without loss of generality that $c_1(T^{\mathbb C}_+ M)
\leq 0$.  Using the transgression formula for the first Chern class, we can
choose a metric $h_+$ on
$T_+^{\mathbb C} M$ such that $\rho(h_+) \leq 0$.  Fix a time $\tau < \tau^*$
and consider the setup as in \S \ref{estimates}.  Since $\rho(h_+) \leq 0$, we
can directly apply the maximum principle to the evolution equation of Lemma
\ref{halfdetev} to conclude an a priori lower bound for $\det g_+$, which of
course is just a lower bound for $g_+$ since $\rank T_+^{\mathbb C} M = 1$. 
Next, from Proposition \ref{fdotest} we obtain a uniform estimate on the ratio
$\frac{\det g_+ \det h_-}{\det h_+ \det g_-}$, which since $g_+$ is bounded
below implies a uniform lower bound for $\det g_-$, which again implies that
$g_-$ is bounded below since $\rank T_-^{\mathbb C} M = 1$.  From Proposition
\ref{mub} we thus conclude uniform upper and lower bounds for $g_t$.  Applying
Theorem \ref{EKregularity} we conclude uniform $C^{\infty}$ estimates for $g_t$
on $[0,\tau)$.  The theorem follows.
\end{proof}
\end{thm}

\begin{proof}[Proof of Theorem \ref{mainthm2}] Let us treat the cases of long
time existence in turn.  As remarked upon in the introduction, generalized
K\"ahler surfaces with commuting complex
structure with rank $1$ were classified in \cite{GaultApost}, building on work
of Beauville \cite{Bville}.  We consider the different parts of this
classification case by case.

\noindent \textbf{Ruled surfaces:}  As shown in (\cite{GaultApost,Bville}), in
this case $(M, J_+)$ is the projectivization of a projectively flat holomorphic
plane bundle over a compact Riemann surface, which we have supposed to be of
genus $g \geq 1$.  As computed in (\cite{Bville} (4.2)), for our (indeed ANY)
splitting of $T^{\mathbb C} M$, one has that the universal cover is a product
with fundamental group acting diagonally, one factor of which must be
biholomorphic to $\mathbb H$.  It follows that either
$c_1(T^{\mathbb C}_{\pm} M) \leq 0$, and so Theorem \ref{halfample} yields the
existence on $[0,\tau^*)$.

\noindent \textbf{Bi-elliptic surfaces:} Arguing as in the case of ruled
surfaces, we know that the universal cover in this setting is biholomorphic to
$\mathbb C \times \mathbb C$.  In particular, both factors satisfy
$c_1(T_{\pm}^{\mathbb C} M) = 0$, and Theorem
\ref{halfample} guarantees the long time existence.

\noindent \textbf{Elliptic fibrations:} As explained in the proof of
(\cite{GaultApost} Theorem 1), the universal cover of $M$ is $\mathbb C \times
\mathbb H$, with the fundamental group acting diagonally by isometries of the
canonical product metric.  Moreover, the splitting into $T_{\pm}^{\mathbb C} M$
corresponds to the obvious splitting of the tangent bundle of $\mathbb C \times
\mathbb H$, and so depending on the orientation, either $c_1(T_{\pm}^{\mathbb
C}) < 0$, and so Theorem \ref{halfample} guarantees the long time existence.

\noindent \textbf{General type:} As shown in (\cite{GaultApost}), the universal
cover of these surfaces is $\mathbb H \times \mathbb H$, with the fundamental
group acting diagonally by biholomorphisms of $\mathbb H$ on each factor. 
Moreover, the splitting of $T^{\mathbb C}M$ corresponds to the natural product
structure of $\mathbb H \times \mathbb H$.  The canonical metric on $\mathbb H$
is invariant under the full biholomorphism group, and so in particular the
canonical product metric descends to $M$, showing that $c_1(T^{\mathbb C}_{\pm}
M) < 0$.  Thus Theorem \ref{halfample} guarantees the long time existence.

\noindent \textbf{Inoue surfaces:} To begin we briefly recall the construction
of Inoue surfaces \cite{Inoue}.  Fix
a matrix $M \in SL(3,\mathbb Z)$ with one real eigenvalue $\ga > 0$ and two
eigenvalues $\gb \neq \bar{\gb}$.  We fix $(a_1,a_2,a_3)$ a real eigenvector for
eigenvalue $\ga$ and $(b_1,b_2,b_3)$ an eigenvector for the eigenvalue $\gb$. 
Let $\mathbb H = \{z |\ \mbox{Im} z > 0 \}$ denote the upper half space, and let
$\gG$
be the group of automorphisms of $\mathbb H \times \mathbb C$ generated by
\begin{align*}
f_0(z,w) = (\ga z, \gb w), \qquad f_j(z,w) = (z + a_j,w + b_j),\quad  1 \leq j
\leq 3.
\end{align*}
The manifold $S_M = \left( \mathbb H \times \mathbb C \right) / \gG$ was
discovered in \cite{Inoue}, and is called an Inoue surface.  This is the
simplest of three classes defined by Inoue, and the only one admitting
generalized K\"ahler structure with commuting complex structures.  In this case
the generalized complex structures come from the corresponding splitting
$\mathbb H \times \mathbb C$.  The two complex structures correspond to choosing
$J_{\pm} = J_{\mathbb H} \oplus (\pm J_{\mathbb C})$.  On these Inoue surfaces
there is a simple pluriclosed metric constructed by Tricerri \cite{Tricerri},
\begin{align} \label{tric}
\gw = \frac{\i}{y^2} dz \wedge d \bar{z} + \i y dw \wedge d \bar{w}.
\end{align}
It is easily checked that in fact this metric is generalized K\"ahler. 
Moreover, we have
\begin{align*}
\rho(g_+) =&\ - \i \del \delb \log y = \frac{1}{4 y^2} dz \wedge d \bar{z},\\
\rho(g_-) =&\ - \i \del \delb \log y^{-2} = - \frac{\i}{2 y^2} dz \wedge d
\bar{z}.
\end{align*}
In particular, $c_1(T_-^{\mathbb C} M) \leq 0$.  Thus we have verified the
hypotheses of Theorem \ref{halfample}, and the long time existence claim in this
case follows.
\end{proof}

\begin{rmk} In the case of bi-elliptic surfaces or surfaces of general type, one
can show convergence of the flow with arbitrary initial data to a
K\"ahler-Einstein metric.  In fact these results hold in a much broader context,
for arbitrary dimension and for solutions to the general pluriclosed flow not
just in the generalized K\"ahler context.  These results will appear in a later
work \cite{SFR}.
\end{rmk}

\begin{rmk} In \cite{Jess} Boling studied homogeneous solutions to the
pluriclosed flow on Inoue surfaces, and found that these always exist for all
time, and after appropriate rescaling converge  in the Gromov-Hausdorff topology
to circles.  It is reasonable to expect that this always happens, and one
approach to answering this would be to establish that the solutions are ``Type
III'' with bounded diameter (after rescaling), and then exploit properties of
the expanding entropy functional (\cite{ST2}) to exhibit the required
convergence, in analogy with the classification of bounded diameter type III
solutions of Ricci flow on three-manifolds given by Lott \cite{Lott}.
\end{rmk}

\bibliographystyle{hamsplain}

\begin{thebibliography}{10}

\bibitem{GaultApost} V. Apostolov, M. Gualtieri, \emph{Generalized K\"ahler
manifolds,  commuting complex structures, and split tangent bundles}, Comm.
Math. Phys. 2007, Vol. 271, Issue 2, 561-575.

\bibitem{Bville} A. Beauville, \emph{Complex manifolds with split tangent
bundle,} in ``Complex analysis and algebraic geometry,'' 61-70, de Gruyter,
Berlin, 2000.

\bibitem{Jess} J. Boling, \emph{Homogeneous solutions of pluriclosed flow on
closed complex surfaces}, \textsf{arXiv:1404.7106}

\bibitem{Buscher} T.H. Buscher, \emph{Quantum corrections and extended
supersymmetry in new sigma models}, Phys. Lett. B \textbf{159}, 127 (1985).

\bibitem{Calabi} E. Calabi, \emph{Improper affine hypersurfaces of convex type
and a generalization of a theorem by K. J\"orgens,} Mich. Math. J. 5 (1958),
105-126.

\bibitem{Calabiconj} E. Calabi, \emph{The space of K\"ahler metrics}, Proc.
Internat. Cong. Math. Amsterdam, 206-207.

\bibitem{CC2} E. Calabi, \emph{On K\"ahler manifolds with vanishing canonical
class}, Algebraic geometry and topology: A symposium in honor of S. Lefschetz,
Princeton Mathematical Series, 78-89.

\bibitem{Evans} L.C. Evans, \emph{Classical solutions of fully nonlinear,
convex, second-order elliptic equations}, Comm. Pure Appl. Math.
35(3):333-363,1982.

\bibitem{GHR} S. Gates, C. Hull, M. Rocek, \emph{Twisted multiplets and new
supersymmetric non-linear $\gs$-models}, Nuclear Physics B248 (1984) 157-186.

\bibitem{Gualtthesis} M. Gualtieri, \emph{Generalized K\"ahler geometry}, Comm.
Math. Phys, 2014, 1-35.

\bibitem{HullSSC} C. Hull, \emph{Superstring compactifications with torsion and
space-time supersymmetry} Turin Superunif. (1985) 347.

\bibitem{HullGCY} C. Hull, U. Lindstrom, M. Rocek, R. Unge, M Zabzine,
\emph{Generalized Calabi-Yau metric and
generalized Monge-Ampere equation}, JHEP August 2010,2010:60.

\bibitem{HitchinGCY} N. Hitchin, \emph{Generalized Calabi-Yau manifolds} Q.J.
Math., 54(3), 281-308, 2003.

\bibitem{Inoue} M. Inoue, \emph{On surfaces of Class $VII_0$}, Invent. Math. 24
(1974), 269-310.

\bibitem{Krylov} N.B. Krylov, \emph{Boundedly inhomogeneous elliptic and
parabolic equations} Izv. Akad. Mauk SSR Ser. Math 46(3):487-523,1982.

\bibitem{Lind} U. Lindstrom, M. Rocek, R. Unge, M. Zabzine, \emph{Generalized
K\"ahler manifolds and off-shell supersymmetry}, Comm. Math. Phys. 269 (2007)
833-849.

\bibitem{Lott} J. Lott, \emph{Dimensional reduction and the long-time behavior
of Ricci flow}, Comment. Math. Helv. 85 (2010), 485-534.

\bibitem{RocekMCY} M. Rocek, \emph{Modified Calabi-Yau manifolds with torsion},
Mirror Symmetry I, 421-429.

\bibitem{SFR} J. Streets, \emph{Pluriclosed flow, Born-Infeld geometry, and
rigidity results for generalized K\"ahler manifolds}, to appear.

\bibitem{STHCF} J. Streets, G. Tian, \emph{Hermitian curvature flow}, JEMS Vol.
13, no. 3 (2011), 601-634.

\bibitem{ST1} J. Streets, G. Tian, \emph{A parabolic flow of pluriclosed
metrics}, Int. Math. Res. Not. 16 (2010), 3101-3133.

\bibitem{ST2} J. Streets, G. Tian, \emph{Regularity results for the pluriclosed
flow}, Geom. \& Top. 17 (2013) 2389-2429.

\bibitem{STGK} J. Streets, G. Tian, \emph{Generalized K\"ahler geometry and the
pluriclosed flow}, Nuc. Phys. B, Vol. 858, Issue 2, (2012) 366-376.

\bibitem{TianZhang} G. Tian, Z. Zhang, \emph{On the K\"ahler-Ricci flow on
projective manifolds of general type} Chinese Ann. Math. Ser. B 27 (2006), no.
2, 179-192.

\bibitem{Tricerri} F. Tricerri, \emph{Some examples of locally conformal
K\"ahler manifolds}, Rend. Sem. Mat. Univ. Pol. Torino 40 (1982), no. 1, 81-92.

\bibitem{Yau} S.T. Yau, \emph{On the Ricci curvature of a compact K\"ahler
manifold and the complex Monge-Ampere equation}, Comm. Pure Appl. Math. 31 (3)
339-411.

\end{thebibliography}

\end{document}